\documentclass[12pt]{article}

\pdfoutput=1

\usepackage{amsmath, amssymb, amsthm, enumerate, fullpage, bm}
\usepackage{verbatim, enumitem, bbm, etoolbox}
\usepackage{latexsym}

\newcommand{\ilim}{{\varprojlim}}

\makeatletter
\newcommand{\dlim@}[2]{%
	\vtop{\m@th\ialign{##\cr
			\hfil$#1\operator@font lim$\hfil\cr
			\noalign{\nointerlineskip\kern1.5\ex@}#2\cr
			\noalign{\nointerlineskip\kern-\ex@}\cr}}%
}
\newcommand{\dlim}{%
	\mathop{\mathpalette\dlim@{\rightarrowfill@\scriptscriptstyle}}\nmlimits@
}
\renewcommand{\varprojlim}{%
	\mathop{\mathpalette\varlim@{\leftarrowfill@\scriptscriptstyle}}\nmlimits@
}
\renewcommand{\varinjlim}{%
	\mathop{\mathpalette\varlim@{\rightarrowfill@\scriptscriptstyle}}\nmlimits@
}
\makeatother

\newcommand{\MM}{{\mathfrak{M}}}

\theoremstyle{definition}
\newtheorem{thm}{Theorem}[section]
\newtheorem{theorem}[thm]{Theorem}

\newtheorem{lemma}[thm]{Lemma}

\numberwithin{subcase}{case}

\theoremstyle{definition}
\newtheorem{definition}[thm]{Definition}
\newtheorem{corollary}[thm]{Corollary}

\def\forkindep{\mathrel{\raise0.2ex\hbox{\ooalign{\hidewidth$\vert$\hidewidth\cr\raise-0.9ex\hbox{$\smile$}}}}}

\def\P{{\cal P}}

\def\Ind{\setbox0=\hbox{$x$}\kern\wd0\hbox to 0pt{\hss$\mid$\hss}
	\lower.9\ht0\hbox to 0pt{\hss$\smile$\hss}\kern\wd0}

\def\Notind{\setbox0=\hbox{$x$}\kern\wd0\hbox to 0pt{\mathchardef
		\nn=12854\hss$\nn$\kern1.4\wd0\hss}\hbox to 0pt{\hss$\mid$\hss}\lower.9\ht0
	\hbox to 0pt{\hss$\smile$\hss}\kern\wd0}

\newcommand{\ptl}{\textrm{ptl}}

\newcommand{\CSS}{\textrm{CSS}}
\newcommand{\css}{\textrm{css}}

\newcommand{\Aut}{\textrm{Aut}}

\def\phi{\varphi}

\def\<{\langle}
\def\>{\rangle}

\makeatletter
\def\blfootnote{\xdef\@thefnmark{}\@footnotetext}
\makeatother

\begin{document}	

	\bibliographystyle{plain}

	\title{Characterizing Borel Isomorphism Among Some Weakly Minimal Trivial Theories}
	\date{\today} 
 \author{Danielle S.\ Ulrich
 \thanks{The author was partially supported
by NSF grants DMS-1855789 and DMS-2154101.}
\\
Department of Mathematics\\University of Maryland
}

	\maketitle

 \begin{abstract}
We characterize having Borel isomorphism relation among some weakly minimal trivial theories, namely the examples of families of finite equivalence relations from \cite{BorelComplexityP}, and tame expansions of crosscutting-equivalence relations. We also prove a dichotomy in Borel complexity for the latter.
\end{abstract}
\section{Introduction}

We are interested in proving dividing lines for countable model theory to characterize Borel complexity. We begin by looking at weakly minimal trivial theories. A countable complete theory $T$ is weakly minimal if $T$ is stable and for all $A \subseteq B$, if $\mbox{tp}(c/B)$ forks over $A$ then $c \in \mbox{acl}(B)$. If $T$ is weakly minimal, then $T$ is trivial if $\mbox{acl}(B) = \bigcup_{b \in B} \mbox{acl}(b)$ for every set $B$. The countable model theory of weakly minimal theories is plausibly tractable. For instance, Vaught's conjecture has been resolved in this context; in fact, Vaught's conjecture has been resolved for superstable theories of finite rank \cite{VaughtSSFR}. 

In \cite{BorelComplexityP} (joint with Laskowski) we introduce a class of weakly minimal trivial theories, which are in some sense the purest such theories:

Let $\mathbb{P}$ be the class of all triples $(P, \leq, \delta)$ where $(P, \leq)$ is a countable partial order, and for all $p \in P$ there are only finitely many $q \in P$ with $q \leq p$, and finally $\delta: P \to \omega \backslash \{0, 1\}$. To each $(P, \leq, \delta) \in \mathbb{P}$ we associate a weakly minimal trivial theory $T(P, \leq, \delta)$ (typically denoted $T_P$ when $\leq, \delta$ are clear from context). The language is $\mathcal{L}_P = \{E_p: p \in P\}$ and $T_P$ asserts that each $E_p$ is an equivalence relation with splitting determined by $\delta$; it follows that each $E_p$ has finitely many classes.

In \cite{BorelComplexityP} we characterize, under the presence of large cardinals, which of these theories are Borel complete. We also introduce the notion of tame expansions, and characterize when $T_P$ has a tame expansion which is Borel complete.

In the present work we are particularly interested in tame expansions of crosscutting equivalence relations. These can alternately be viewed as products of finite structures, for a peculiar notion of product: suppose $(M_n: n < \omega)$ are finite structures in disjoint relational languages $\mathcal{L}_n$, each of size at least two. Let $\mathcal{L} =  \bigcup_n \mathcal{L}_n \cup \{E_n\}$ where $E_n$ is a new equivalence relation. Turn $\prod_n M_n$ into an $\mathcal{L}$-structure (also denoted $\prod_n M_n$) as follows: for each $m$ let $\prod_n M_n \models E_m(a, b)$ if and only if $a(m) = b(m)$, and for $R(x_i: i < i_*) \in \mathcal{L}_m$ let $\prod_n M_n \models R(a_i: i < i_*)$ if and only if $M_m \models R(a_i(m): i < i_*)$. We let $T_{\overline{M}}$ be the theory of $\prod_n M_n$, which is weakly minimal trivial. These theories $T_{\overline{M}}$ are exactly the tame expansions of crosscutting equivalence relations.

Crosscutting equivalence relations can be recovered as the case when each $M_n$ is a set with no extra structure. In that case, in \cite{BorelComplexityP} we showed that $T_{\overline{M}}$ is Borel complete if and only if $|M_n|: n < \omega$ is bounded below $\omega$. When $M_n$ is allowed to have extra structure, this need no longer be the case; indeed, let $M_n$ be a set of size $n$, equipped with a cyclic order. Then it is easy to check that $T_{\overline{M}}$ has Borel isomorphism relation, and thus cannot be Borel complete.

We show:

\begin{theorem}
Suppose $M_n: n < \omega$ are finite structures in disjoint relational languages, each of size at least two. Then either $\|T_{\overline{M}}\| < \beth_{\omega_1}$ or else $T_{\overline{M}}$ is Borel complete.
\end{theorem}

Here $\|T_{\overline{M}}\|$ is the potential cardinality of $T_{\overline{M}}$, a cardinal invariant of Borel complexity.

We move on to investigate back-and-forth complexity. The Scott rank of a model $M$, denoted $\mbox{sr}(M)$, has multiple definitions. We work with the definition most convenient for our purposes, as in \cite{Marker} \cite{Gao}, namely $\mbox{sr}(M)$ is the least ordinal $\alpha$ such that for all $\overline{a}, \overline{b} \in  M^{<\omega}$, if $(M, \overline{a}) \equiv_\alpha (M, \overline{b})$ then $(M, \overline{a}) \equiv_{\infty} (M, \overline{b})$. For a sentence $\Phi \in \mathcal{L}_{\omega_1 \omega}$ we have that $\cong_\Phi$ is Borel if and only if the Scott ranks of countable models of $\Phi$ are bounded below $\omega_1$. An easy absoluteness argument gives that $\cong_\Phi$ is Borel if and only if the Scott ranks of {\em all} models of $\Phi$ are bounded below $\omega_1$. Let ${\mbox{sr}}(\Phi)$ denote the least ordinal $\alpha$, if it exists, such that for all $M \models \Phi$, $\mbox{sr}(M) < \alpha$. Thus $\cong_\Phi$ is Borel if and only if ${\mbox{sr}}(\Phi) < \omega_1$.

In all of our examples $T$, we will have ${\mbox{sr}}(T) \leq (2^{\aleph_0})^+$, since every model of $T$ is back-and-forth equivalent to a model of size at most $2^{\aleph_0}$. We consider it interesting to compute $\mbox{sr}(\Phi)$, especially when $\cong_\Phi$ is non-Borel, since in this case $\mbox{sr}(\Phi)$ can be viewed as measuring how badly $\cong_\Phi$ fails to be Borel.

We show the following, where an action of a group $G$ on a set $X$ is free if for all $g \in G$ and $x \in X$, if $g\cdot x = x$ then $g = 1^G$.

\begin{theorem}\label{TameExpThm}
    Suppose $M_n: n < \omega$ are finite structures in disjoint relational languages, each of size at least two. Let $G_n = \mbox{Aut}(M_n)$, a finite permutation group. Then $\cong_{T_{\overline{M}}}$ is Borel if and only if for all but finitely many $n$, the action of $G_n$ on $M_n$ is free.

    If $\cong_{T_{\overline{M}}}$ is non-Borel then $\mbox{sr}(T_{\overline{M}}) = (2^{\aleph_0})^+$.
\end{theorem}

We use this to return to the original class of examples $T_P$:

\begin{theorem}\label{PThm}
    Suppose $(P, \leq, \delta) \in \mathbb{P}$. Then $\cong_{T_P}$ is Borel if and only if there is a finite downward-closed $Q \subseteq P$ such that $P \backslash Q$ is an antichain on which $\delta$ is identically $2$.
\end{theorem}

The proof proceeds by finding benchmarks $(P_i, \leq, \delta_i) \in \mathbb{P}$ for $i < 4$ such that whenever $\cong_{T_Q}$ is non-Borel then some $T_{P_i} \leq_B T_Q$. We do not know for general $P$ how to compute ${\mbox{sr}}(T_P)$, but at least for the benchmarks we know ${\mbox{sr}}(T_{P_1}) = {\mbox{sr}}(T_{P_2}) = (2^{\aleph_0})^+$ (in fact, they are both tame expansions of crosscutting equivalence relations, so the preceding theorem applies) and ${\mbox{sr}}(T_{P_0}) = {\mbox{sr}}(T_{P_3}) = \omega_1$.

As a particular example, let $T_k$ be the theory of crosscutting equivalence relations $E_n$ where $E_n$ has $k$-many classes. Then $\cong_{T_2}$ is Borel (this is easy), but $\cong_{T_3}$ is not Borel.

\section{Preliminaries}

\subsection{Borel Complexity}
In \cite{FS} Friedman and Stanley introduced the notion of Borel complexity. Namely, for each sentence $\Phi$ of $\mathcal{L}_{\omega_1 \omega}$ consider the Polish space $\mbox{Mod}(\Phi)$ of models of $\Phi$ of universe $\omega$. Say that $\Phi$ is Borel reducible to $\Psi$, and write $\Phi \leq_B \Psi$, if there is a Borel map $f: \mbox{Mod}(\Phi) \to \mbox{Mod}(\Psi)$ such that for all $M, N \in \mbox{Mod}(\Phi)$, $M \cong N$ if and only if $f(M) \cong f(N)$. In practice the models we deal with will rarely have universe $\omega$; however the process of choosing isomorphic copies with universe $\omega$, and then checking Borelness of the map, is always routine. Say that $\Phi$ is Borel complete if it is maximal under $\leq_B$.

To every model $M$ we have associated its canonical Scott sentence $\css(M)$, which characterizes $M$ up to back-and-forth equivalence; when $M$ is countable, $\css(M)$ thus characterizes $M$ up to isomorphism. 

If $\Phi$ is a sentence of $\mathcal{L}_{\omega_1 \omega}$ then we let $\CSS(\Phi)_{\ptl}$ denote the class of all potential canonical Scott sentences of $\Phi$, i.e. sentences of $\mathcal{L}_{\omega_1 \omega}$ which, in some forcing extension, are of the form $\css(M)$ for some (countable) $M$. We let $\|\Phi\|$, the potential cardinality of $\Phi$, denote the cardinality of the class $\CSS(\Phi)_{\ptl}$, possibly $\infty$. In \cite{URL} we show that if $\Phi \leq_B \Psi$ then the Borel reduction induces an injection from $\CSS(\Phi)_{\ptl}$ to $\CSS(\Psi)_{\ptl}$, hence $\|\Phi\| \leq \| \Psi\|$. Thus $\|\Phi\|$ is a cardinal invariant of the Borel complexity of $\Phi$.

\subsection{Scott Ranks}
The following definition of $\equiv_\alpha$ and Scott rank is as in \cite{Marker}:

\begin{definition}
Suppose $M$, $N$  are structures in the same language $\mathcal{L}$ and $\overline{a} \in M^{<\omega}$ and $\overline{b} \in N^{<\omega}$. Say that $(M, \overline{a}) \equiv_0 (N, \overline{b})$ if $\mbox{qftp}^M(\overline{a}) = \mbox{qftp}^N(\overline{b})$ (in particular they have the same length). For $\alpha$ an ordinal, say that $(M, \overline{a}) \equiv_{\alpha+1} (N, \overline{b})$ if for all $a \in M$ there is $b \in M$ with $(M, \overline{a}a) \equiv_{\alpha} (N, \overline{b} b)$ and vice versa. For a limit ordinal $\delta$ say that $(M, \overline{a}) \equiv_\delta (N, \overline{b})$ if for all $\alpha < \delta$, $(M, \overline{a}) \equiv_\alpha (N, \overline{b})$. Say that $(M, \overline{a}) \equiv_{\infty} (N, \overline{b})$ if $(M, \overline{a}) \equiv_\alpha (N, \overline{b})$ for all ordinals $\alpha$ (equivalently, for all $\alpha < |M|^+ + |N|^+$).

Define the Scott rank of $M$, $\mbox{sr}(M)$, to be the least ordinal $\alpha$ such that for all $\overline{a}, \overline{b} \in M^{<\omega}$ if $(M, \overline{a}) \equiv_\alpha (M, \overline{b})$ then $\overline{a} \equiv_\infty (M, \overline{b})$, so always $\mbox{sr}(M) < |M|^+$. If $\Phi \in \mathcal{L}_{\omega_1 \omega}$ then let $\mbox{sr}(\Phi)$ denote the least ordinal $\alpha$, if it exists, such that for all $M \models \Phi$, $\mbox{sr}(M) < \alpha$.
\end{definition}

There are finer ways to measure back-and-forth complexity than the rank given above; for instance there is the categoricity rank of \cite{RobusterScottRanks} and the Scott complexity of \cite{ScottComplexity}, both of which are phrased in terms of countable models but which can be generalized.

The following (with the parenthetical included) is Theorem 12.2.4 of \cite{Gao}. Here $\cong_\Phi \subseteq \mbox{Mod}(\Phi) \times \mbox{Mod}(\Phi)$ is the isomorphism relation on models of $\Phi$ with universe $\omega$; it is always analytic. The parenthetical can be removed by an absoluteness argument.
\begin{lemma}\label{SRLemma1}
Suppose $\Phi$ in $\mathcal{L}_{\omega_1 \omega}$. Then $\cong_\Phi$ is Borel if and only if there is $\alpha < \omega_1$ such that for all (countable) $M \models \Phi$, $\mbox{sr}(M) < \alpha$.
\end{lemma}
So $\cong_\Phi$ is Borel if and only if $\mbox{sr}(\Phi) <\omega_1$.

For countable structures, the following follows from Theorem 12.1.8 of \cite{Gao}. The general case has the same proof.
\begin{lemma}\label{SRLemma2}
Suppose $M$, $N$ are $\mathcal{L}$ structures and $\alpha = \mbox{sr}(M)$ and $M \equiv_{\alpha + \omega} N$. Then $\mbox{sr}(N) = \alpha$ and $M \equiv_{\infty \omega} N$.
\end{lemma}
We collect together some easy lemmas which are well-known but for which we could not find a direct reference.

\begin{lemma}\label{SRLemma3}
    Suppose $M$ is an $\mathcal{L}$-structure, $t(\overline{x})$ is an $\mathcal{L}$-term, and $(M, \overline{a}) \equiv_\alpha (M, \overline{b})$. Then $(M, \overline{a} t(\overline{a})) \equiv_\alpha (M, \overline{b} t(\overline{b}))$.
\end{lemma}
\begin{proof}
    Straightforward induction on $\alpha$ (simultaneously for all terms $t$).
\end{proof}

\begin{lemma}\label{SRLemma4}
    Suppoose $M$ is an $\mathcal{L}$-structure and $\alpha$ is an infinite ordinal and $(M, \overline{a}) \equiv_\alpha (M, \overline{b})$. Suppose $a \in M$ is the unique realization of its quantifier-free type over $\overline{a}$. Then there is a unique $b \in M$ with $\mbox{qftp}(b, \overline{b}) = \mbox{qftp}(a,\overline{a})$, and for this $b$ we have $(M, \overline{a}a) \equiv_\alpha (M, \overline{b} b)$.
\end{lemma}
\begin{proof}
    A similar induction to the preceding lemma, but taking extra care at the base case $\alpha = \omega$.

    Suppose $(M, \overline{a}) \equiv_\omega (M, \overline{b})$ and $a \in M$ is the unique realization of its quantifier-free type over $\overline{a}$. Since $(M, \overline{a}) \equiv_2 (M, \overline{b})$, we have that there is a unique $b \in M$ with $\mbox{qftp}^M(\overline{a}a) = \mbox{qftp}^M(\overline{b} b)$. For each $n$, since $(M, \overline{a}) \equiv_{n+1} (M, \overline{b})$, there is some $b' \in M$ with $(M, \overline{a}a) \equiv_n (M, \overline{b}b')$; by uniqueness of $b$, $b= b'$. Thus $(M, \overline{a}a) \equiv_n (M, \overline{b} b)$ for all $n < \omega$, so $(M, \overline{a}a) \equiv_\omega (M, \overline{b})$ as desired.
\end{proof}

\begin{lemma} \label{SRLemma5}
    Suppose $M, N$ are structures in the same language and $\overline{a}, \overline{b}$ are tuples from $M$ and $\overline{c}, \overline{d}$ are tuples from $N$, with $\overline{a}$ and $\overline{b}$ of the same length and $\overline{c}, \overline{d}$ of the same length. Let $M' = M(\overline{a})$ and let $N' = N(\overline{c})$, i.e. add $\overline{a}$ and $\overline{c}$ as constants. Then for all ordinals $\alpha$, $(M, \overline{a} \overline{b}) \equiv_\alpha (N, \overline{c} \overline{d})$ if and only if $(M', \overline{b}) \equiv_\alpha (N', \overline{d})$.
\end{lemma}
\begin{proof}
    Straightforward induction on $\alpha$.
\end{proof}

\begin{lemma}\label{SRLemma6}
    Suppose $M$ is a structure and $\overline{c} \in M^n$. Let $N$ be the expansion of $M$ where we add $\overline{c}$ as a sequence of constants. Then $\mbox{sr}(N) \leq \mbox{sr}(M) \leq \mbox{sr}(N) + \omega + n$.
\end{lemma}
\begin{proof}
    It is routine to check that $\mbox{sr}(N) \leq \mbox{sr}(M)$. For the reverse implication, write $\alpha = \mbox{sr}(N)$ and suppose $(M, \overline{a}) \equiv_{\alpha+\omega+n}(M, \overline{b})$. Then there is $\overline{d} \in M^n$ with $(M, \overline{a}, \overline{c}) \equiv_{\alpha + \omega} (M, \overline{b}, \overline{d})$. Let $M' = M(\overline{d})$, i.e. add $\overline{d}$ as constants. Using Lemma~\ref{SRLemma5} we have that $(N, \overline{a}) \equiv_{\alpha + \omega}(M', \overline{b})$. By Lemma~\ref{SRLemma2}, we have that $(N, \overline{a}) \equiv_{\infty \omega} (M', \overline{b})$. Hence $(M, \overline{a}) \equiv_{\infty \omega} (M, \overline{b})$ as desired.
\end{proof}

We mention the mechanism by which all of our examples of Borel isomorphism relation arise.

\begin{definition}
    Suppose $M$ is a countable structure. Then a base is a subset $B \subseteq M$ such that distinct elements of $M$ have distinct first-order types over $B$.
\end{definition}

\begin{lemma}\label{SRLemma8}
    Suppose $\Phi$ is a sentence of $\mathcal{L}_{\omega_1 \omega}$ such that every countable model of $\Phi$ admits a finite base. Then $\cong_\Phi$ is Borel, in fact $\mbox{sr}(\Phi) \leq \omega \cdot 3$.
\end{lemma}
\begin{proof}
    It suffices to show that if $M$ is a countable structure admitting a finite base $B$, then $\mbox{sr}(M) < \omega \cdot 3$. Let $N$ be the expansion of $M$ where we add $B$ as constants. Then $\mbox{sr}(M) \leq \mbox{sr}(N) + \omega + n < \mbox{sr}(N) + \omega \cdot 2$ by Lemma~\ref{SRLemma2} so it suffices to show $\mbox{sr}(N) \leq \omega$. But this follows from the fact that distinct elements of $N$ have distinct first-order types. 
\end{proof}

\subsection{$S_\infty$ Dividing}
\begin{definition}
    (From \cite{BorelComplexityP}) Suppose $M$ is a countable structure and $D_n \subseteq M$. Say that $D_n: n < \omega$ are absolutely indiscernible sets if the $D_n$'s are disjoint and for all $\sigma \in \mbox{Sym}(\omega)$, there is $g \in \mbox{Aut}(M)$ with each $g(D_n) = D_{\sigma(n)}$.

    (Due to Hjorth \cite{Hjorth}) Suppose $G, H$ are Polish groups. Say that $G$ divides $H$ if there is some closed $K \leq H$ and some continuous surjective homomorphism $\pi: K \to G$. We are particularly interested in the case when $G = S_\infty$ is the group of permutations of $\omega$.
\end{definition}

The following theorem is implicit in \cite{Allison}:
\begin{theorem}
Suppose $M$ is a countable structure. Then $S_\infty$ divides $\mbox{Aut}(M)$ if and only if $M$ admits absolutely indiscernible sets.
\end{theorem}



We pull some definitions from \cite{Allison}.
\begin{definition}
Suppose $M$ is a countable structure. Then $\mbox{cl}: \mathcal{P}(M) \to \mathcal{P}(M)$ is an invariant closure operation if the following hold for all $A, B \subseteq M$:

\begin{itemize}
    \item $A \subseteq \mbox{cl}(A)$;
    \item $\mbox{cl}(A) \subseteq \mbox{cl}(B)$ for all $A \subseteq B$;
    \item $\mbox{cl}^2 = \mbox{cl}$;
    \item $\mbox{cl}$ is invariant under automorphisms of $M$.
\end{itemize}

Given an invariant closure operation $\mbox{cl}$, we have an associated independence relation $\forkindep$ defined by $A \forkindep_C B$ if $A \cap \mbox{cl}(BC) \subseteq \mbox{cl}(C)$ and $B \cap \mbox{cl}(AC) \subseteq \mbox{cl}(C)$. $\mbox{cl}$ is disjointifying if for all finite sets $A, B \supseteq C$, there is some $A' \cong_C A$ with $A' \forkindep_C B$. $\mbox{cl}$ is trivial if $\mbox{cl}(\emptyset) = M$.
\end{definition}

Allison proves the following \cite{Allison}:
\begin{theorem}
 $S_\infty$ divides $\mbox{Aut}(M)$ if and only if there is a nontrivial disjointifying closure operation. 
\end{theorem}

A trivial modification of his proof gives: 

\begin{corollary}
Suppose $M$ is a countable structure and $\mbox{cl}$ is a nontrivial invariant disjointifying closure relation on $M$. Suppose $a \in M$ and $a \not \in \mbox{cl}(\emptyset)$. Then we can find absolutely indiscernible sets $D_n: n< \omega$ such that for all $b \in D_n$, $(M, b) \cong (M, a)$. 
\end{corollary}

We will need the following minor observation.

\begin{lemma}\label{clLemma}
    Suppose $\overline{a} \in M^{<\omega}$ and $\mbox{cl}$ is a nontrivial disjointifying closure relation on $M$. Suppose $\mbox{cl}(\overline{a}) \not= M$. Let $N$ be $M(\overline{a})$ (i.e. add constants for $\overline{a}$) and define $\mbox{cl}'(A) := \mbox{cl}(A \overline{a})$ for all $A \subseteq M$. Then $\mbox{cl}'$ is a nontrivial disjointifying closure operation on $N$.
\end{lemma}
\begin{proof}
    Straightforward.
\end{proof}

\section{The Examples}

We recall the examples from \cite{BorelComplexityP}.

Say that the poset $P$ is down-finite if for all $p \in P$, $P_{\leq p}$ is finite. Let $\mathbb{P}$ be the set of all triples $(P, \leq, \delta)$ where $(P, \leq)$ is a down-finite countable poset and $\delta: P \to \omega \backslash \{0, 1\}$. Usually $\leq$ and $\delta$ will be clear from context and we will just write $P$.

Suppose $(P, \leq, \delta) \in \mathbb{P}$. Let $\mathcal{L}_P$ be the language with binary relations $E_p$ for each $p \in P$. We define an $\mathcal{L}_P$-structure $\mathfrak{M}_P$, an $\mathcal{L}_P$-theory $T_P$, and an infinitary sentence $\Phi_P$ of $(\mathcal{L}_{P})_{\omega_1 \omega}$. Namely, let $\mathfrak{M}_P$ be the structure with universe $\prod_{p \in P} \delta(p)$ where we interpret $\mathfrak{M}_P \models E_p(f, g)$ if and only if $f(q) = g(q)$ for all $q \leq p$. Let $T_P$ be the theory axiomatized by the following:

\begin{itemize}
 \item  Each $E_q$ is an equivalence relation;
 \item  If $q'\le q$, then $E_q$ refines $E_{q'}$ and, moreover, letting $E_{<q}(x,y):=\bigwedge_{q'<q} E_{q'}(x,y)$,
 then $E_q$ partitions each $E_{<q}$-class into precisely $\delta(q)$ classes;
 \item If $Q\subseteq P$ is downward closed and finite,  then for every sequence $(a_q:q\in Q)$ satisfying
 $E_{q'}(a_{q'},a_q)$ for all $q'\le q$, there is $a^*$ such that $E_q(a^*,a_q)$ for every $q\in Q$.
 \end{itemize}

 Let $\Phi_P$ assert $T_P$ along with the infinitary sentence asserting that for all $a, b$, if $E_p(a,b)$ for all $p \in P$ then $a = b$.

 When $P$ is finite $T_P$ is incomplete, but when $P$ is infinite then $T_P$ is eliminates quantifiers, and generates the complete first-order theory of $\mathfrak{M}_P$; and $\mathfrak{M}_P \models \Phi_P$. Further, every model of $\Phi_P$ embeds as a dense, equivalently, elementary submodel of $\mathfrak{M}_P$. When $P$ is infinite then $T_P$ is weakly minimal trivial (when $P$ is finite, then every completion is weakly minimal trivial). 

 We will want the following characterization of the countable models of $T_P$. For $\Phi$ any sentence of $\mathcal{L}_{\omega_1 \omega}$ we let $\mathcal{C}(\Phi)$ denote the infinitary sentence describing $\omega$-colorings of models of $\Phi$ (formally, we add unary predicates for the colors). 

The following is Lemma 4.3 of \cite{BorelComplexityP}:
 \begin{lemma}
     For all $(P, \leq, \delta) \in \mathbb{P}$, $T_P$ is Borel equivalent to $\mathcal{C}(\Phi_P)$.
 \end{lemma}

In practice it is easier to work with $\mathcal{C}(\Phi_P)$.

We now turn to tame expansions.

\begin{definition}  \label{tamedef} { Suppose $(P,\le,\delta)\in\P$ and $M\models T_P$.  
Fix any $p\in P$.  A subset $S\subseteq M^n$ is {\em $E_p$-invariant} if, for all $(a_0,\dots,a_{n-1}),(b_0,\dots,b_{n-1})\in M^n$, if $M\models \bigwedge_{i<n}E_p(a_i,b_i)$, then $[\overline{a}\in S\leftrightarrow \overline{b}\in S]$.  
A subset $S\subseteq M^n$ is {\em tame} if $S$ is $E_p$-invariant for some $p\in P$.

Let $\mathcal{L}^+=\mathcal{L}_P\cup\{S_i(\overline{x}_i):i\in I\}$ where each $S_i$ is $n_i$-ary.  An expansion $M^+$ of $M \models T_P$  is a {\em tame expansion} if
the interpretation of every $S_i(\overline{x}_i)$ is a tame subset of $M^{n_i}$. A tame expansion of $T_P$ is the complete theory of a tame expansion $M^+$ of some $M \models T_P$.
}
\end{definition}

Every tame expansion of $T_P$ is weakly minimal trivial. Moreover, we have the following Fact 8.2 of \cite{BorelComplexityP}:

\begin{lemma}\label{tameFacts}
Suppose $(P,\le,\delta) \in \mathbb{P}$.
\begin{enumerate}
    \item For all $M \preceq N \models T_P$ and for all tame expansions $M^+$ of $M$, there is a unique tame expansion $N^+$ of $N$ with $M^+ \subseteq N^+$.
    \item If $M \preceq N \models T_P$ and $M^+ \subseteq N^+$ are tame expansions then $M^+ \preceq N^+$.
    \item Suppose $T$ is a tame expansion of $T_P$, with $P$ infinite. Then $T$ is Borel equivalent to $\mathcal{C}(\Phi_P \land T)$. (If $P$ is finite then $T$ has exactly one countable model up to isomorphism.)
\end{enumerate}
\end{lemma}

Expanding on (3), we will want that Scott ranks are preserved as well:

\begin{lemma}
Suppose $(P, \leq, \delta) \in \mathbb{P}$ and $T$ is a tame expansion of $T_P$. Then $\mbox{sr}(\mathcal{C}(\Phi_P \land T)) \leq \mbox{sr}(T) \leq \omega + \mbox{sr}(\mathcal{C}(\Phi_P \land T))$.    
\end{lemma}
\begin{proof}

    Given $M \models T$ let $E_P$ denote the equivalence relation on $M$ defined by $E_P(a, b)$ if and only if $E_p(a, b)$ for all $p \in P$. $M/E_P$ is naturally a model of $\Phi_P \land T$; we define a coloring $c: M/E_P \to \omega$ via $c([a]_{E_P}) = 0$ if $[a]_{E_P}$ is infinite, otherwise $c([a]_{E_P}) = |[a]_{E_P}| + 1$. Every model of $\mathcal{C}(\Phi_P \land T)$ is isomorphic to $(M/E_P, c)$ for some $M \models T$, so it suffices to show that $\mbox{sr}((M/E_P, c)) \leq \mbox{sr}(M) \leq \omega+ \mbox{sr}((M/E_P, c))$. Let $N$ denote $(M/E_p, c)$.

    We first note by induction on $\alpha < \omega_1$ that for all tuples of distinct elements $(a_i: i < n), (b_i: i < n)$ from $M$, if $(N, [\overline{a}]_{E_P}) \equiv_\alpha (N, [\overline{b}]_{E_P})$ then $(M, \overline{a}) \equiv_\alpha (M, \overline{b})$. For $\alpha = 0$ we use that $\overline{a}$ and $\overline{b}$ are tuples of distinct elements; $\alpha$ limit is immediate. Suppose $(N, [\overline{a}]_{E_P}) \equiv_{\alpha+1} (N, [\overline{b}]_{E_P})$. We need to show $(M, \overline{a}) \equiv_{\alpha+1} (M, \overline{b})$. Suppose say $a \in M$, we need to find $b \in M$ with $(M, \overline{a}a) \equiv_\alpha (M, \overline{b}b)$. If $E_P(a, a_i)$ for some $i < n$ then use the fact that $|[a_i]_{E_P}| = |[b_i]_{E_P}|$ since $c([a_i]_{E_P}) = c([b_i]_{E_P})$. If no such $i < n$ exists then choose some $[b]_{E_P} \in N$ such that $(N, [\overline{a}a]_{E_P}) \equiv_\alpha (N, [\overline{b}b]_{E_P})$. Then by the inductive hypothesis, $(M, \overline{a} a) \equiv_\alpha (M, \overline{b} b)$ as desired.

    Next we note by induction on $\alpha <\omega_1$ that for all tuples $(a_i: i < n), (b_i: i < n)$ from $M$, if $(M, \overline{a}) \equiv_{\omega+\alpha} (M, \overline{b})$ then $(N, [\overline{a}]_{E_P}) \equiv_{\alpha} (N, [\overline{a}_{E_P}])$. The case $\alpha = 0$ uses that if $(M, \overline{a}) \equiv_\omega (M, \overline{b})$ then each $|[a_i]_{E_P}| = |[b_i]_{E_P}|$. The case of $\alpha$ limit is immediate. So suppose $(M, \overline{a}) \equiv_{\omega+\alpha+1} (M, \overline{b})$, and $ a\in M$ is given; we need to find $b \in M$  with $(N, [\overline{a}a]_{E_P}) \equiv_{\alpha} (N, [\overline{b}b]_{E_P})$. Choose $b \in M$ such that $(M, \overline{a}a) \equiv_\alpha (N, \overline{b}b)$ and apply the inductive hypothesis.

    The Lemma follows from these two inductions.
\end{proof}

Note then that if $\mbox{sr}(\mathcal{C}(\Phi_P\land T)) \geq \omega \cdot \omega$ then $\mbox{sr}(\mathcal{C}(\Phi_P \land T)) = \mbox{sr}(T)$.

We observe:
\begin{lemma}
Suppose $(P, \leq, \delta) \in \mathbb{P}$. Then:
    \begin{enumerate}
        \item If $P$ is infinite, then every tame expansion $T$ of $T_P$ is of the form $\mbox{Th}(\mathfrak{M}^+)$, where $\mathfrak{M}^+$ is a tame expansion of $\mathfrak{M}_P$;
        \item Two tame expansions of $\mathfrak{M}_P$ are elementarily equivalent if and only if they are isomorphic;
        \item Suppose $\mathfrak{M}^+$ is a tame expansion of $\mathfrak{M}_P$. Then every model of $\Phi_P \land T$ elementarily embeds into $\mathfrak{M}^+$.
    \end{enumerate}
\end{lemma}
\begin{proof}
    (1) Let $T$ be given, say $T = \mbox{Th}(M^+)$ where $M^+$ is a tame expansion of $M\models T_P$. Then $M \equiv \mathfrak{M}_P$, since both are models of the complete theory $T_P$. Let $\mathfrak{C}$ be the monster of $T_P$. Then by Lemma~\ref{tameFacts}(1) there is a tame expansion $\mathfrak{C}^+$ of $\mathfrak{C}$ with $M^+ \subseteq \mathfrak{C}^+$. Let $\mathfrak{M}^+$ be the substructure of $\mathfrak{C}^+$ with universe $\mathfrak{M}$. Then by Lemma~\ref{tameFacts}(2), both $M^+$ and $\mathfrak{M}^+$ are elementarily substructures of $\mathfrak{C}^+$, hence they are elementarily equivalent as desired.

    (2) Let $\mathfrak{M}^+, \mathfrak{M}^{\dagger}$ be two elementarily equivalent tame expansions of $\mathfrak{M}_P$, say in the language $\mathcal{L}^+$. For each $p \in P$ let $X_p = [\mathfrak{M}_P]_{E_p}$ and let $\mathcal{L}^+_p$ be the set of all $R \in \mathcal{L}^+$ which are $E_p$-invariant; so given $q \in P$, $E_q \in \mathcal{L}^+_p$ if and only if $q \leq p$. Let $M_p$ be the $\mathcal{L}^+_p$-structure with universe $X_p$, where $M_p \models R(\overline{x})$ if and only if for some or any $\overline{a} \in \mathfrak{M}_P$ with $[\overline{a}]_{E_p} = \overline{x}$, we have $\mathfrak{M}^+ \models R(\overline{a})$; let $N_p$ be defined similarly but with $\mathfrak{M}^{\dagger}$ replacing $\mathfrak{M}^+$. Since $\mathfrak{M}^{\dagger}$ and $\mathfrak{M}^+$ are elementarily equivalent, so are $M_p$ and $N_p$; hence $M_p$ and $N_p$ are isomorphic. By Konig's lemma, we can find isomorphisms $f_p: M_p \cong N_p$ which cohere, in the sense that for all $q < p \in P$, and for all $[a]_{E_p} \in X_p$, writing $f_p([a]_{E_p}) = [b]_{E_p}$, we have that $f_q([a]_{E_q}) = [b]_{E_q}$. Then $(f_p: p \in P)$ induces an isomorphism from $\mathfrak{M}^+$ to $\mathfrak{M}^{\dagger}$.

    (3) Suppose $M^+ \models \Phi_P \land T$ is given, so $M^+$ is the tame expansion of $M \models \Phi_P$. We can suppose $M$ is an elementary substructure of $\mathfrak{M}$; let $\mathfrak{M}^{\dagger}$ be the tame expansion of $\mathfrak{M}$ with $M^+ \preceq \mathfrak{M}^{\dagger}$. Apply (2) to $\mathfrak{M}^+$ and $\mathfrak{M}^{\dagger}$.
\end{proof}

Finally, we consider tame expansions of crosscutting equivalence relations. To be precise, let $\mathbb{P}_{cc}$ denote the set of all triples $(P, \leq, \delta) \in \mathbb{P}$ such that $P$ is an infinite antichain. These are what we mean by crosscutting equivalence relations. We aim to show that tame expansions of crosscutting equivalence relations are as described in the introduction. We set some notation.

 Let $\mathbb{M}$ be the family of all sequences $(M_n: n < \omega)$ where $M_n$ is a finite relational structure of size at least two and the languages of the $M_n$'s are disjoint. For each $\overline{M} \in \mathbb{M}$ let $(P, \leq, \delta) \in \mathbb{P}_{cc}$ be defined by letting $(P, \leq)$ be the infinite antichian with universe $\omega$, and let $\delta(n) = |M_n|$. Let $\mathfrak{M}$ be the isomorphic copy of $\mathfrak{M}_P$ with underlying universe $\prod_n X_n$, where $X_n$ is the universe of $M_n$. Expand $\mathfrak{M}$ to an $\mathcal{L}_P \cup \bigcup_n \mathcal{L}(M_n)$-structure, denoted $\prod_n M_n$, by letting, for $R \in \mathcal{L}_m$, $\prod_n M_n \models R(a_i:i < i_*)$ if and only if $M_m \models R(a_i(m): i < i_*)$. We see at once that $\prod_n M_n$ is a tame expansion of $\mathfrak{M}$, and every tame expansion of $\mathfrak{M}$ arises in this way.

 For $\overline{M} \in \mathbb{M}$ and $(P, \leq, \delta)$ as above let $T_{\overline{M}} = \mbox{Th}(\prod_n M_n)$, let $\Phi_{\overline{M}} = T_{\overline{M}} \land \Phi_P$. So $T_{\overline{M}}$ is weakly minimal trivial, and is Borel equivalent to $\mathcal{C}(\Phi_{\overline{M}})$, and every model of $\Phi_{\overline{M}}$ can be embedded into $\prod_n M_n$. Moreover, if $\cong_{T_{\overline{M}}}$ is not Borel (equivalently $\cong_{\mathcal{C}(\Phi_{\overline{M}})}$ is not Borel) then $\mbox{sr}(T_{\overline{M}}) = \mbox{sr}(\mathcal{C}(\Phi_{\overline{M}}))$.

\section{Borel Complexity of Tame Expansions of Crosscutting Equivalence Relations}

In this section we analyze the Borel complexity of $T_{\overline{M}}$ for $\overline{M} \in \mathbb{M}$. We shall need some setup.
\begin{definition}
Suppose $M$ is a countable structure and $E$ is an invariant equivalence relation on $M$ and $D_n: n < \omega$ are absolutely indiscernible sets. Say that $D_n: n < \omega$ is $E$-large if each $D_n$ is $E$-saturated (note this implies $[D_n]_E \cap [D_m]_E = \emptyset$ for $n \not= m$, since $D_n$ and $D_m$ are disjoint). Say that $(D_n: n < \omega)$ is $E$-small if all the $D_n$'s are contained in a single $E$-class.
\end{definition}

\begin{lemma}
Suppose $M$ is a countable structure with $S_\infty$ dividing $\mbox{Aut}(M)$. Suppose $E$ is an invariant equivalence relation on $M$. Then there are absolutely indiscernible sets $(D_n: n < \omega)$ which are either $E$-small or else $E$-large.
\end{lemma}
\begin{proof}
Let $N$ be obtained from $M$ by adding a sort for $M/E$ and adding the projection map $\pi: M \mapsto M/E$. Since $S_\infty$ divides $\mbox{Aut}(M)$, also $S_\infty$ divides $\mbox{Aut}(N)$ (they are isomorphic automorphism groups). Thus we can find a nontrivial disjointifying closure operation $\mbox{cl}$ on $N$. There are now two cases:

Case 1: suppose $M/E \subseteq \mbox{cl}(\emptyset)$. Then since $\mbox{cl}(\emptyset) \not= N$, we must be able to find some $ \alpha \in M/E$ such that $\pi^{-1}(\alpha) \not \subseteq \mbox{cl}(\alpha)$. Fix some such $\alpha$ and let $a \in \pi^{-1}(\alpha) \backslash \mbox{cl}(\alpha)$. Let $N' := N(\alpha)$ (add a constant for $\alpha$) and let $\mbox{cl}'$ be defined by $\mbox{cl}'(A) = \mbox{cl}(A\alpha)$. By Lemma~\ref{clLemma}, $\mbox{cl}'$ is a nontrivial disjointifying closure operation on $N'$, and $a \not \in \mbox{cl}'(\emptyset)$, so $N'$ admits absolutely indiscernible sets $D_n: n < \omega$ with each $b \in D_n$ satisfying $(N', b) \cong (N', a)$. In particular each $D_n \subseteq M$ and each $b \in D_n$ has $\pi(b) = \alpha$. Thus $D_n: n < \omega$ is $E$-small.

Case 2: otherwise, pick $\alpha \in M/E \backslash \mbox{cl}(\emptyset)$. Pick absolutely indiscernible sets $D'_n: n < \omega$ such that for all $\beta \in D'_n$, $(N, \beta) \cong (N, \alpha)$. Let $D_n = \pi^{-1}(D'_n)$ for each $n < \omega$. Then $D_n: n < \omega$ is $E$-large. 
\end{proof}

\begin{definition} \label{cross-cutting} {A countable structure $M$ {\em admits cross-cutting absolutely indiscernible sets}
if there are $\Aut(M)$-invariant equivalence relations $E_0,E_1$ and families of sets $\mathcal{D}_*^0=\{D^0_n:n\in\omega\}$,
$\mathcal{D}^1_*=\{D^1_m:m\in\omega\}$ satisfying
\begin{enumerate}
    \item  For all $a\in\mathcal{D}^0_*$, $b\in \mathcal{D}^1_*$, there is $c\in M$ such that $M\models E_0(c,a)\land E_1(c,b)$;
    \item  For all distinct $n\neq n'$, $a\in D^0_n$, $a'\in D^0_{n'}$ implies $M\models \neg E_0(a,a')$ (and dually for 
    $\mathcal{D}^1_*$ and $E_1$.)
    \item  For all pairs $\sigma_0,\sigma_1\in \mbox{Sym}(\omega)$ there is some $\tau\in\Aut(M)$ such that for all
    $n,m\in\omega$, $\tau[D^0_n]=D^0_{\sigma_0(n)}$ and $\tau[D^1_m]=D^1_{\sigma_1(m)}$.
\end{enumerate}
}
\end{definition}

The following is Theorem 2.1.3 of \cite{BorelComplexityP}, where $\mathcal{C}(M)$ is the class of all $\omega$-colorings of $M$:

\begin{theorem} \label{binary}   Suppose a countable structure $M$ admits cross-cutting absolutely indiscernible sets.
Then $\mathcal{C}(M)$ is Borel complete.  
\end{theorem}

Putting this together, we obtain:

\begin{theorem}
Suppose $\overline{M} = (M_n: n < \omega) \in \mathbb{M}$. Suppose $M \models \Phi_{\overline{M}}$ is countable. If $S_\infty$ divides $\mbox{Aut}(M)$, then $\mathcal{C}(\Phi_{\overline{M}})$ is Borel complete.
\end{theorem}
\begin{proof}
We can suppose $M$ is an elementary submodel of $\prod_n M_n$. 
 Let $E$ be the invariant equivalence relation on $\prod_n M_n$ defined by: $a E b$ if and only if for all but finitely many $n$, $a(n) = b(n)$.
 
    By the preceding lemma, we can find some absolutely indiscernible sets $D_n: n < \omega$ for $M$ which are either $E$-small or $E$-large. $E$-small is impossible, since each $E$-class has cli automorphism group, hence $S_\infty$ cannot divide it. Thus $D_n: n < \omega$ are $E$-large. In particular, for all $n \not= m$, for all $a \in D_n, b \in D_m$, $a$ and $b$ differ infinitely often; call this infinite set $I_{a, b}$.

    It is now straightforward to find infinite disjoint subsets $I_0, I_1 \subseteq \omega$ such that for all $n \not= m$, $a \in D_n, b \in D_m$ we have $I_{a, b} \cap I_0$ and $I_{a, b} \cap I_1$ are both infinite. We can also arrange $I_0 \cup I_1 = \omega$.

    Let $N$ be the countable set of all $a \in \prod_n M_n$ such that for each $i < 2$, $a \restriction_{I_i} \in \{b \restriction_{I_i}: b \in M\}$. For $i < 2$ let $E_i$ be the equivalence relation on $N$ given by: $E_i(a, b)$ if $a \restriction_{I_i} = b \restriction_{I_i}$. For $i < 2$ and $n < \omega$ let $D^i_n$ be the set of all $a \in N$ such that $a \restriction_{I_i} \in \{b \restriction_{I_i}: b \in D_n\}$. Then it is straightforward to check that $(E_0, E_1, D^i_n: i < 2, n < \omega)$ witness that $N$ admits crosscutting absolutely indiscernible sets, so $\mathcal{C}(\Phi_{\overline{M}})$ is Borel complete by Theorem~\ref{binary}.
\end{proof}
\begin{corollary}
Either $T_{\overline{M}}$ is Borel complete or else $\|T_{\overline{M}}\| < \beth_{\omega_1}$. 
\end{corollary}
\begin{proof}
    Suppose $T_{\overline{M}}$ is not Borel complete, i.e. $\mathcal{C}(\Phi_{\overline{M}})$ is not Borel complete. Then $S_\infty$ does not divide the automorphism group of any countable model of $\Phi_{\overline{M}}$, by the preceding theorem. Hence also $S_\infty$ does not divide the automorphism group of any countable model of $\mathcal{C}(\Phi_{\overline{M}})$. By Theorem 5.5 of \cite{expansions}, no expansion of $\mathcal{C}(\Phi_{\overline{M}})$ is Borel complete. By Corollary 6.1 of \cite{expansions}, $\|\mathcal{C}(\Phi_{\overline{M}}) \| < \beth_{\omega_1}$, i.e. $\|T_{\overline{M}}\| < \beth_{\omega_1}$.
\end{proof}

We discuss the possibilities for the potential cardinality $\|T_{\overline{M}}\| = \|\mathcal{C}(\Phi_{\overline{M}})\|$. We give a simple lower bound in a more general context:

\begin{theorem}\label{LowerBound}
    Suppose $(P, \leq, \delta) \in \mathbb{P}$ with $P$ infinite and $T$ is a tame expansion of $T_P$. Then "countable sets of reals" Borel embeds into $T$, hence $\|T\| \geq \beth_2$.
\end{theorem}
\begin{proof}
    
    Let $\mathfrak{M}$ be an expansion of $\mathfrak{M}_P$ to a model of $T$. Recall that $T$ is Borel equivalent to $\mathcal{C}(T \land \Phi_P)$, the models of which can be identified with countable dense subsets of $\mathfrak{M}$ equipped with a coloring into $\omega$. Let $M_0 \subseteq \mathfrak{M}$ be a countable dense set. Enumerate $M_0 = \{a_n: n < \omega\}$. Let $X = \mathfrak{M} \backslash M_0$, a standard Borel space. For each countable $Y \subseteq X$ let $(M_Y, c_Y) \models \mathcal{C}(T \land \Phi_P)$ be defined as follows: $M_Y = M_0 \cup Y$, and $c_Y(a_n) = n+1$ and $c_Y(b) = 0$ for all $b \in Y$. Then $Y  \mapsto (M_Y, c_Y)$ is a map from $[X]^{\aleph_0}$ to countable models of $\mathcal{C}(T \land \Phi_P)$. This map is easily seen to be Borel in the codes, and distinct elements of $[X]^{\aleph_0}$ get mapped to nonisomorphic models of $\mathcal{C}(T \land \Phi_P)$.
\end{proof}

Thus for $\overline{M} \in \mathbb{M}$ we have $\|T_{\overline{M}}\| \geq \beth_2$. We conjecture that this value is attained whenever $T_{\overline{M}}$ is not Borel complete.

\section{Complicated Inverse Systems of Abelian Groups}
Our next item of interest is to understand the Scott ranks of models of tame expansions of crosscutting equivalence relations. The key to this is coding inverse systems of abelian groups, which we discuss in this section.

By an inverse system of abelian groups, we mean a system $\mathbf{A} = (A_p, \pi_{pq}: p < q \in P)$ where $P$ is a directed poset and each $A_p$ is an abelian group and $\pi_{pq}: A_q \to A_p$ are commuting homomorphisms (we do not require the homomorphisms to be onto). For convenience we require that the $A_p$'s be disjoint.  

We define a rank function $\mbox{rnk}^{\mathbf{A}}: \bigsqcup_p A_p \to \mbox{ON} \cup \{\infty\}$ as follows: suppose $a \in A_p$. Then put $\mbox{rnk}^{\mathbf{A}}(a) \geq \alpha$ if and only if for each $q > p$ and for each $\beta < \alpha$ there is $b \in A_q$ with $\mbox{rnk}^{\mathbf{A}}(b) \geq \beta$ and $\pi_{pq}(b) = a$. If $P$ is countable, then for $a \in A_p$ we have $\mbox{rnk}^{\mathbf{A}}(a) = \infty$ if and only if there is some $\mathbf{a} \in \ilim(\mathbf{A})$ with $\mathbf{a}(p) = a$; but if $P$ is uncountable then the forward direction can fail. Let $\mbox{rnk}(\mathbf{A})$ denote the least ordinal $\alpha$ such that for all $p \in P$ and all $a \in A_p$, if $\mbox{rnk}^{\mathbf{A}}(a) < \infty$ then $\mbox{rnk}^{\mathbf{A}}(a) < \alpha$.

Suppose $I_*$ is an infinite index set and $C_i: i \in I_*$ are abelian groups. Then a $\overline{C}$-system is an inverse system of groups $\mathbf{A} = (A_I, \pi_{IJ}: I \subseteq J  \in \mathbf{I})$ where $\mathbf{I}$ is a suborder of $(\mathcal{P}(I_*), \subseteq)$ and for each $I \in \mathbf{I}$, $A_I$ is a subgroup of $\prod_{i \in I} C_i$ and each $\pi_{IJ}$ is the natural restriction map (so we will usually omit it).

\begin{lemma}
    Suppose $C$ is a nonzero cyclic group and $I_*$ is an infinite index set; let $C_i = C$ for each $i \in I_*$. Suppose $\alpha_* < (2^{\aleph_0})^+$. Then there exists a $\overline{C}$-system $\mathbf{A} = (A_I: I \in \mathbf{I})$ with $\mbox{rnk}(\mathbf{A})> \alpha_*$. 
\end{lemma}

\begin{proof}

Let $T$ be a well-founded tree of size at most continuum, with $\mbox{rnk}^T(0^T)>\omega \cdot \alpha_*$. Here $\mbox{rnk}^T: T \to \mbox{ON}$ is the usual foundation rank, so $\mbox{rnk}^T(s) \geq \alpha$ if and only if for all $\beta < \alpha$ there is $t > s$ with $\mbox{rnk}^T(t) \geq \beta$. 

Let $U$ be the set of nonempty finite downward-closed $u \subset T$. For each $t \in T$ let $\mbox{succ}^+(t)$ denote $t$ along with all immediate successors of $T$. For each $u \in U$ let $\mbox{succ}^+(u) = \bigcup_{t \in u} \mbox{succ}^+(t)$. 

Recall that two sets are almost disjoint if they have finite intersection. There exist continuum-sized almost disjoint families of subsets of any given infinite set.

Let $(J_t: t \in T)$ be almost disjoint infinite subsets of $I_*$. For each $t \in T$ choose $(J_{ts}: s \in \mbox{succ}^+(t))$ almost disjoint infinite subsets of $J_t$.

 For each $u \in U$ let $I_u := \bigcup_{t \in u} J_t$, a subset of $I_*$.  Then $I_u \subseteq I_v$ for $u \subseteq v$, and for $u \not \subseteq v$ we have $I_u \backslash I_v$ is infinite, since it contains all but finitely many elements of $J_{tt}$ where $t \in u \backslash v$. 
 
 Let $\mathbf{I} = \{I_u: u \in U\}$, a suborder of $(\mathcal{P}(I_*), \subseteq)$ isomorphic to $(U, \subseteq)$.

Say that $f: X \to Y$ is eventually constant of value $c$ if for all but finitely many $x \in X$, $f(x) = c$. For each $u \in U$, let $A_{I_u} \leq C^{I_u}$ be the set of all functions $f$ satisfying the following:

\begin{itemize}
    \item For each $t \in u$ and for each $s \in \mbox{succ}^+(t)$, $f \restriction_{J_{ts}}$ is eventually constant;
    \item Define $\sigma_f: \mbox{succ}^+(u) \to C$ as follows. Let $\sigma_f(0^T)$ be the eventual value of $f \restriction_{J_{0^T0^T}}$; for all other $s \in \mbox{succ}^+(u)$ let $\sigma_f(s)$ be the eventual value of $f \restriction_{J_{s^-s}}$ where $s^-$ is the immediate predecessor of $s$;
    \item For all $t \in u$, $\sum_{s \in \mbox{succ}^+(t)} \sigma_f(s) = 0$ (in particular only finitely many terms are nonzero).
\end{itemize}

Each $A_{I_u}$ is a subgroup of $C^{I_u}$, and for $u \subseteq v$ the restriction of $f \in A_{I_v}$ to $I_u$ is in $A_{I_u}$. Thus we have a $\overline{C}$-system $\mathbf{A}$. It remains to show $\mbox{rnk}(\mathbf{A}) > \alpha_*$.

It is convenient to reindex $A_u = A_{I_u}$ for $u \in U$. It is easily checked that for all $u \subseteq v \in U$ and for all $f \in A_{v}$, writing $g := f \restriction_{I_u}$, we have that $\sigma_g = \sigma_f \restriction_{\mbox{succ}^+(u)}$. 

For $\alpha$ an ordinal, say that $f \in A_{u}$ is $\alpha$-strong if for all $s \in \mbox{succ}^+(u)$, if $\sigma_f(s) \not= 0$ then $\mbox{rnk}^T(s) \geq \alpha$. 

\noindent \textbf{Claim 1.} For all $f \in A_{u}$, and for all ordinals $\alpha$, if $\mbox{rnk}^{\mathbf{A}}(f) \geq \alpha$ then $f$ is $\alpha$-strong. 

Proof: By induction on $\alpha$. Suppose $f \in A_{u}$ is not $\alpha$-strong; we show $\mbox{rnk}^{\mathbf{A}}(a) < \alpha$.  Pick $s \in \mbox{succ}^+(u)$ with $\sigma_f(s) \not= 0$ and $\mbox{rnk}^T(s) < \alpha$. Write $\beta = \mbox{rnk}^T(s)$ and write $v = u \cup \{s\}$. For any extension $g$ of $f$ to $A_{v}$, we must have $\sigma_g(t) \not= 0$ for some immediate successor $t$ of $s$, hence necessarily $g$ is not $\beta$-strong, hence by the inductive hypothesis $\mbox{rnk}^{\mathbf{A}}(g) < \beta$. Hence $\mbox{rnk}^{\mathbf{A}}(f) < \alpha$.

\noindent \textbf{Claim 2.} Suppose $u \subsetneq v$ are in $U$ and $v= u \cup \{s\}$ for some $s \in T$. Suppose $f \in A_{u}$ is $\alpha+1$-strong. Then there is $g \in A_{v}$ extending $f$ such that $g$ is $\alpha$-strong.

Proof: We have $I_v = I_u \cup J_s$, and $J_s \cap I_u$ is finite.

If $\sigma_f(s) = 0$ then let $g \in A_{v}$ be the extension of $f$ with $g(i) = 0$ for all $i \in I_v \backslash I_u$.

So suppose $\sigma_f(s) \not= 0$. Then $\mbox{rnk}^T(s) > \alpha$ so we can find some immediate successor $t$ of $s$ with $\mbox{rnk}^T(t) \geq \alpha$. Let $g \in A_v$ be the extension of $f$ with $g(i) = -\sigma_f(s)$ for all $i \in J_{st} \backslash I_u$, and $g(i) = 0$ for all $i \in I_v \backslash (J_{st} \cup I_u)$.

\noindent \textbf{Claim 3.} Suppose $f \in A_u$ is $\omega \cdot \alpha$-strong. Then $\mbox{rnk}^{\mathbf{A}}(f) \geq \alpha$.

Proof: follows from Claim 2.

Now we finish. Since $\mbox{rnk}^T(0^T) > \omega \cdot \alpha_*$, we can find an immediate successor $s$ of $0^T$ with $\mbox{rnk}^T(s) \geq \omega \cdot \alpha_*$. Let $b$ be a nonzero element of $C$. Reasoning as in Claim 2, we can find $f \in A_{\{0^T\}}$ with $\sigma_f(0^T) = a$ and $\sigma_f(s) = -a$ and $\sigma_f(t) =0 $ for all other immediate successors $t$ of $0^T$. Then $f$ is $\omega \cdot \alpha_*$-strong, so $\mbox{rnk}^{\mathbf{A}}(f) \geq \alpha_*$; but also $f$ is not $\mbox{rnk}^T(s)+1$-strong, so $\mbox{rnk}^{\mathbf{A}}(f) < \infty$, as desired.
\end{proof}

\begin{theorem}
    Suppose $I_*$ is a countably infinite index set and $C_i$ is a nonzero abelian group for each $i \in I_*$. Suppose $\alpha_*< (2^{\aleph_0})^+$. Then there is a $\overline{C}$-system $\mathbf{A} = (A_I: I \in \mathbf{I})$ with $\mbox{rnk}({\mathbf{A}}) > \alpha_*$. 
\end{theorem}
\begin{proof}

Each $C_i$ contains a nonzero cyclic group, so it suffices to consider the case where each $C_i$ is cyclic. Let $k_i = |C_i|$, so $2 \leq k_i \leq \omega$. If $k_i$ is constant on an infinite set then apply the preceding Lemma, so suppose this is not the case.

    Choose infinite disjoint sets $K_n: n <\omega$ from $I_*$ such that $k_i$ is unbounded below $\omega$ on each $K_n$. For $n < \omega$ let $D_n$ be an infinite cylic subgroup of $\prod_{i \in K_n} C_i$, for instance let $D_n$ be the cyclic group generated by $(g_i: i \in K_n)$ where $g_i$ generates $C_i$. So each $D_n \cong \mathbb{Z}$.
    
    By the preceding lemma, let $\mathbf{B} = (B_J: J \in \mathbf{J})$ be a $(D_n: n < \omega)$-system with $\mbox{rnk}(\mathbf{B}) > \alpha_*$. For each $J \in \mathbf{J}$ let $I_J = \bigcup_{n \in J} K_n$. Let $\mathbf{I} = \{I_J: J \in \mathbf{J}\}$.

    Finally, for each $I \in \mathbf{I}$, say $I = I_J$ for $J \in \mathbf{J}$, let $A_{I}$ be the set of all $a \in \prod_{i \in I} C_i$ such that for each $n \in J$, $a \restriction_{K_n} \in D_n$ and $(a \restriction_{K_n}: n \in J) \in B_J$. Then $\mathbf{A} = (A_I: I \in \mathbf{I})$ is as desired.
\end{proof}
\section{Scott Ranks of Tame Expansions of Crosscutting Equivalence Relations}

Let $\overline{M} = (M_n: n < \omega) \in \mathbb{M}$, so each $M_n$ is a finite relational structure of size at least two and the languages of the $M_n$'s are disjoint. We aim to classify when $\cong_{T_{\overline{M}}}$ is Borel, i.e. when $\mbox{sr}(T_{\overline{M}}) < \omega_1$.

For each $n$ let $G_n = \mbox{Aut}(M_n)$; so $G_n$ acts on $M_n$, but perhaps not transitively. The action of $G_n$ on $M_n$ is called free if for all $g \in G_n$, if $g$ fixes some point of $M_n$ then $g$ is the identity.

Let $I_*$ be the set of all $n < \omega$ such that the action of $G_n$ on $M_n$ is not free.

\begin{theorem}
$\cong_{T_{\overline{M}}}$ is Borel if and only if $I_*$ is finite. In fact, if $I_*$ is infinite then $T_{\overline{M}}$ has models of Scott rank arbitrarily large below $(2^{\aleph_0})^+$, so $\mbox{sr}(T_{\overline{M}}) = (2^{\aleph_0})^+$.
\end{theorem}
The rest of this section is a proof. As usual we work with $\mathcal{C}(\Phi_{\overline{M}})$ rather than $T_{\overline{M}}$. Note that $\mbox{sr}(\mathcal{C}(\Phi_{\overline{M}})) \leq (2^{\aleph_0})^+$ because every model of $\mathcal{C}(\Phi_{\overline{M}})$ has size at most continuum.

First suppose $I_*$ is finite. We claim that every $(M, c) \models \mathcal{C}(\Phi_{\overline{M}})$ has a finite base, which suffices by Lemma~\ref{SRLemma8}. It suffices to show that every $M \models \Phi_{\overline{M}}$ has a finite base. Suppose $M$ is given; we can suppose $M \preceq \prod_n M_n$. Choose $\overline{a} \in M^{<\omega}$ nonempty such that for all $n \in I_*$, $\{a_i(n): i < |\overline{a}|\}$ covers all of $M_n$. This is clearly possible. We show $\overline{a}$ is a base for $M$. Suppose $a, b$ are two elements of $M$ with $\mbox{tp}^M(a/\overline{a}) = \mbox{tp}(b/\overline{a})$. We show $a = b$; it suffices to show $a(n) = b(n)$ for all $n< \omega$. It is clear that $a(n) = b(n)$ for all $n \in I_*$; so fix $n \not \in I_*$ and we show $a(n) = b(n)$. We have $\mbox{tp}^{M_n}(a(n)/a_0(n)) = \mbox{tp}^{M_n}(b(n)/a_0(n))$, so there is $g \in G_n$ with $g(a_0(n), a(n)) = (a_0(n), b(n))$, and since the action of $G_n$ on $M_n$ is free this implies $a(n) = b(n)$ as desired.

The other direction aims to apply the theorem from the preceding section. We suppose $I_*$ is infinite. Let $\alpha_*< (2^{\aleph_0})^+$; we produce a model of $\mathcal{C}(\Phi_{\overline{M}})$ of Scott rank at least $\alpha_*$. Actually we won't need the coloring; i.e. we produce a model of $\Phi_{\overline{M}}$ of Scott rank at least $\alpha_*$.

For each $n \in I_*$ pick $o_n \in M_n$ and $g_n \in G_n$ such that $g_n(o_n) = o_n$ and $g_n$ is not the identity; we can arrange that $g_n$ is of prime order $p_n$. Let $d_n \in M_n$ be such that $g_n(d_n) \not= d_n$; necessarily then $g_n^k(d_n) \not= d_n$ for all $k < p_n$. For each $n \in I$ let $C_n$ be the cyclic subgroup of $G_n$ generated by $g_n$; we write $C_n$ additively.
 
For each $n \not \in I$ let $o_n \in M_n$ be arbitrary. Let $\overline{o} \in \prod_n M_n$ be the sequence $(o_n: n < \omega)$. Given $f \in \prod_n M_n$ let $\mbox{supp}(f) = \{n < \omega: f(n) \not= o_n\}$. 

By the preceding theorem, let $\mathbf{A} = (A_I: I \in \mathbf{I})$ be a $(C_n:n \in I_*)$-system with $\mbox{rnk}(\mathbf{A}) > \alpha_*$. By padding we can suppose each $I \in \mathbf{I}$ is infinite and $\bigcup \mathbf{I} = I_*$. 


For each $a \in A_I$ let $F_I(a) \in \mbox{Aut}(\prod_{n < \omega} M_n) = \prod_{n < \omega} G_n$ be defined by $F_I(a) \restriction_{M_n} = $ the identity map unless $n \in I$, in which case $F_I(a)\restriction_{M_n} = a(n) \in C_n$. Clearly $F_I$ is a homomorphism of groups. We write $F_I(a)(f) = a+f$.

For each $I \in \mathbf{I}$ let $f_I \in \prod_n M_n$ be defined by $f_I(n) = o_n$ for $n \not \in I$, and $f_I(n) = d_n$ for $n \in I$. So $\mbox{supp}(f_I) = I$. Note that $A_I$ acts regularly on $A_I + f_I$, since for all $n \in I$ and all $g \in C_n$ nonzero, $g(d_n) \not= d_n$. 

Let $M$ be the submodel of $\prod_n M_n$ consisting of all $f \in \prod_n M_n$ of finite support, along with the union of all $A_I + f_I$ for $I \in \mathbf{I}$. Let $N = M(\overline{o})$, i.e. add $\overline{o}$ as a constant. By Lemma~\ref{SRLemma6} it suffices to show $\mbox{sr}(N) \geq \alpha_*$.

\begin{lemma}\label{SRTELemma1}
For each ordinal, and for each $a \in A_I$, if $(N, f_I) \equiv_{\omega+\alpha} (N, a + f_I)$ then $\mbox{rnk}^{\mathbf{A}}(a) \geq \alpha$.
\end{lemma}
\begin{proof}
    By induction on $\alpha$. Suppose we have verified this for all $\beta < \alpha$ and $a \in A_I$ is given with $(N, f_I) \equiv_{\alpha} (N, a + f_I)$. We need to show $\mbox{rnk}^{\mathbf{A}}(a) \geq \alpha$. Suppose $\beta < \alpha$ and $J \supseteq I$ are given. We have that there is some $x$ with $(N, f_I, f_J) \equiv_{\beta} (N, a+f_I,x)$. Since $f_J$ has support $J$, so must $x$ (since we added $\overline{o}$ as a constant), so $x \in A_J + f_J$, write it as $a' + f_J$. By the inductive hypothesis, $\mbox{rnk}^{\mathbf{A}}(a') \geq \beta$.

Recall that for $n \in \omega$ $E_n \in \mathcal{L}(\prod_n M_n)$ is the equivalence relation with $E_n(f, g)$ if and only if $f(n) = g(n)$. We have $M \models E_n(f_I, f_J)$ for all $n \in I$, hence $M \models E_n(a+f_I, a'+f_J)$ for all $n \in I$, which means that $a' \restriction_{I} + f_I = a + f_I$. Since $A_I$ acts regularly on $A_I + f_I$, this implies $a' \restriction_I = a$, and since $\mbox{rnk}^{\mathbf{A}}(a') \geq \beta$ this witnesses $\mbox{rnk}^{\mathbf{A}}(a) \geq \alpha$.
\end{proof}

For each $I \in \mathbf{I}$ let $M_I$ be the set of all $e \in M$ such that $\mbox{supp}(e) \cap I_* \subseteq I$. So $A_I+ M_I = M_I$, and for $I \subseteq J$ and $b \in B_J$ and $e \in M_I$, $b+ e = (b \restriction_{I}) + e$.

\begin{lemma}\label{SRTELemma2}
Suppose $I \in \mathbf{I}$ and $\overline{e} = (e_i: i < i_*)$ is a tuple from $M_I$. Let $a \in A_I$ have $\mbox{rnk}^{\mathbf{A}}(a) \geq \alpha$. Then  
$(N, (a+e_i: i < i_*)) \equiv_{\alpha} (N, \overline{e})$.
\end{lemma}
\begin{proof}
By induction on $\alpha$. It suffices to show this for $M$ rather than $N$ because each element of $A_I$ fixes $\overline{o}$. The case $\alpha = 0$ follows from the fact that $a$ acts on $\prod_n M_n$ by automorphisms. The case of $\alpha$ limit is immediate. So suppose $\alpha$ is a successor ordinal, say $\alpha = \beta+1$. Suppose $\mbox{rnk}^{A}(a) \geq \alpha$ and suppose $\overline{e}$ is a tuple from $M_I$. The situation is symmetric in $\overline{e}$ and $(a+ e_i: i < i_*)$ (since $(-a+(a+e_i): i < i_*) = \overline{e}$) so it suffice so thow that for all $\beta < \alpha$ and for all $e \in M$ there is $e' \in M$ with  $((M, c), \overline{e}e) \equiv_{\beta} ((M, c), (a + e_i: i < i_*)e')$. Choose $J \supseteq I$ with $e \in M_J$ (possible using that $\mathbf{I}$ is directed and that $\bigcup \mathbf{I} = I_*$). Choose $a' \in A_J$ with $a' $ extending $a$ and $\mbox{rnk}^{\mathbf{A}}(a') \geq \beta$. Then $(a' + e_i: i < i_*) = (a+ e_i: i < i_*)$, so we can let $e' = a' + e$ and use the indutive hypothesis.
\end{proof}

With these lemmas in hand, we now finish. By choice of $\mathbf{A}$ we can choose some $a \in A_I$ with $\alpha_* \leq \mbox{rnk}^{\mathbf{A}}(a) < \infty$. By Lemma~\ref{SRTELemma1}, we have $(N, f_I) \not \equiv_{\infty} (N, a+f_I)$, but by Lemma~\ref{SRTELemma2}, we have $(N, f_I) \equiv_{\alpha_*} (N, a+f_I)$.

\section{Scott Ranks of Families of Finite Equivalence Relations}

\begin{definition}
    Suppose $(P, \leq, \delta) \in \mathbb{P}$. Then say that $(P, \leq, \delta)$ is nearly binary crosscutting if all but finitely many $p \in P$ are maximal and have $\delta(p)  = 2$. Equivalently, $P$ is nearly binary crosscutting if there is a finite downward-closed $Q \subseteq P$ such that $P \backslash Q$ is an antichain on which $\delta$ is identically $2$. 
\end{definition}

We aim to prove Theorem~\ref{PThm}, i.e. show that for $(P, \leq, \delta) \in \mathbb{P}$, $\cong_{T_P}$ is Borel if and only if $P$ is nearly binary crosscutting. As usual it suffices to work with $\mathcal{C}(\Phi_P)$ rather than $T_P$. To begin:

\begin{lemma}
    Suppose $(P, \leq, \delta) \in \mathbb{P}$ is nearly binary crosscutting. Then $\cong_{\mathcal{C}(\Phi_P)}$ is Borel.
\end{lemma}
\begin{proof}
Suppose $(M, c) \models \mathcal{C}(\Phi_P)$; it suffies by Lemma~\ref{SRLemma8} to show that $(M, c)$ has a finite base. Let $Q \subseteq P$ be finite and downward closed such that $P \backslash Q$ is an antichain on which $\delta$ is identically $2$. 

Let $E_Q$ denote $\bigwedge_{q \in Q} E_q$; this is a quantifier-free definable equivalence relation on $M$ with finitely many classes.

Choose $\overline{a} = (a_i: i < i_*) \in M^{<\omega}$ such that $\{[a_i]_{E_{Q}}: i < i_*\}$ exhausts $[M]_{E_Q}$. It suffices to show $\overline{a}$ is a base for $M$. Suppose $a, b \in M$ have $\mbox{tp}^M(\overline{a} a) = \mbox{tp}^M(\overline{a} b)$; it suffices to show $E_p(a, b)$ for all $p \in P$. Choose $i < i_*$ with $E_Q(a, a_i)$ and so $E_Q(b, a_i)$.

Now if $p \in Q$ then $a E_p a_i E_p b$ and we are done, so suppose $p \in P \backslash Q$ and suppose towards a contraction $\lnot E_p(a, b)$. Then we must have $\lnot E_p(a, a_i)$ and $\lnot E_p(b, a_i)$. But then $\{a, b, a_i\}$ are all $E_{<p}$-equivalent and all $E_p$-inequivalent, contradicting $\delta(p) =2$.
\end{proof}

We aim to show the converse. We reduce it to checking a finite list of examples. Towards this, we want the following Lemmas.

\begin{lemma}\label{ReductionP}
Suppose $(P, \leq, \delta) \in \mathbb{P}$ and $Q \subseteq P$ ($Q$ need not be downward closed). Then $\mathcal{C}(\Phi_Q) \leq_B \mathcal{C}(\Phi_P)$.
\end{lemma}
\begin{proof}
Given $(M, c) \models \mathcal{C}(\Phi_Q)$ with $M \preceq \mathfrak{M}_Q$ let $(M_*, c_*) \models \Phi_P$ be defined via: $M_* \preceq \mathbb{M}_P$, and given $f \in \mathbb{M}_P$, $f \in M_*$ if and only if $f \restriction_Q \in M$ and $f \restriction_{P \backslash Q}$ is nonzero finitely often. Let $c_*(f) = 0$ if $f \restriction_{P \backslash Q}$ is not identically $0$, and otherwise let $c_*(f) = c(f \restriction_Q)+1$. We claim that $(M, c) \mapsto (M_*, c_*)$ is the desired reduction.

First suppose $(M, c) \cong (M', c')$ are models of $\mathcal{C}(\Phi_Q)$ with $M, M' \preceq \mathfrak{M}_Q$, we show $(M_*, c_*) {\cong} (M'_*, c'_*)$. Let $\sigma \in \mbox{Aut}(\mathfrak{M}_P)$ take $(M, c)$ to $(M', c')$ (see Fact 4.2(3) of \cite{BorelComplexityP}). Let $\tau \in \mbox{Aut}(\mathfrak{M}_P)$ be defined by sending $f \in \mathfrak{M}_P$ to $\sigma(f \restriction_Q) \cup f\restriction_{P \backslash Q}$. Then $\tau$ is easily seen to be an isomorphism taking $(M_*, c_*)$ to $(M'_*, c'_*)$.

To finish it suffices to show we can recover $(M, c)/\cong$ from $(M_*, c_*)$. But working in $M_*$, let $M'$ be the set of all $f$ such that $c_*(f) \not= 0$, and define $c'(a) = c_*(a) -1$ for each $a \in M'$. Then viewing $M'$ as an $\mathcal{L}_Q$-structure we get that $(M', c') \cong (M, c)$ via the map sending $a$ to $a \restriction_Q$, as desired.
\end{proof}

\begin{lemma}\label{ReductionDelta}
Suppose $(P, \leq, \delta) \in \mathbb{P}$ and $\delta' \leq \delta$ pointwise, i.e. $\delta'(p) \leq \delta(p)$ for all $p \in P$. Then $\mathcal{C}(\Phi(P, \leq, \delta')) \leq_B \mathcal{C}(\Phi(P, \leq, \delta))$.
\end{lemma}
\begin{proof}
Our normal notation of $T_P, \Phi_P, \mathfrak{M}_P$ is ambiguous here, so we abbreviate $\Phi = \Phi(P, \leq, \delta), \mathfrak{M} = \mathfrak{M}(P, \leq, \delta)$, and $\Phi' = \Phi(P, \leq, \delta')$, $\mathfrak{M}' = \mathfrak{M}(P, \leq, \delta')$.

Let $\mathcal{I}$ be the set of all finite intersections $\bigcap_{i < n} P_{\not \geq p_i}$ for $\overline{p} \in P^{<\omega}$. We construe $P \in \mathcal{I}$ as being the empty intersection. Given $f \in \mathfrak{M}$ let $I(f) \subseteq P$ be the set of all $p \in P$ such that $f(q) < \delta'(q)$ for all $q \leq p$.

Given $(M, c) \models \mathcal{C}(\Phi')$ with $M \preceq \mathfrak{M}'$, let $(M_*, c_*) \models \Phi(P, \delta)$ with $M_* \preceq \mathfrak{M}$ be defined as follows: suppose $f \in \mathfrak{M}$. Then $f \in M_*$ if and only if $I(f) \in \mathcal{I}$ and $f \restriction_{I(f)}$ can be extended to an element of $M$, and $f \restriction_{P \backslash I(f)}$ is nonzero only finitely often. Define $c_*(f) = c(f) + 1$ if $f \in M$ (i.e. if $I(f) = P$), and $c_*(f) = 0$ otherwise. We claim that $(M, c) \mapsto (M_*, c_*)$ is the desired Borel reduction.

First suppose $(M, c) \cong (M', c')$ are models of $\Phi'$; let $\sigma \in \mbox{Aut}(\mathfrak{M}')$ take $(M, c)$ to $(M', c')$. Then define $\tau \in \mbox{Aut}(\mathfrak{M})$ via $\tau(f) = \sigma(f \restriction_{I(f)}) \cup f \restriction_{P \backslash I(f)}$. It is straightforward to check that $\tau$ is an automorphism, i.e. that $\tau$ is a bijection preserving the $E_p$'s, and that $\tau$ takes $(M_*, c_*)$ to $(M'_*, c'_*)$.

To finish it suffices to show that $(M, c)/\cong$ can be recovered from $(M_*, c_*)/\cong$, but this is clear (look at the elements of nonzero color, and subtract one from their colors).
\end{proof}

Now we introduce our benchmark equivalence relations.

\begin{definition}
    Let $(P_0, \leq, \delta_0) \in \mathbb{P}$ be such that $P_0$ is an $\omega$-chain and $\delta_0$ is identically $2$. So $T_{P_0}$ is just REF(bin) from \cite{URL}.

    Let $(P_1, \leq, \delta_1) \in \mathbb{P}$ be such that $P_1$ is an infinite antichain and $\delta_1$ is constantly $3$.

    Let $(P_2, \leq, \delta_2) \in \mathbb{P}$ be such that $P_2 = \{p_n, q_n: n < \omega\}$ where the only relations are $p_n < q_n$ for each $n$, and where $\delta_2$ is identically $2$.

    Let $(P_3, \leq, \delta_3) \in \mathbb{P}$ be such that $P_3 = \{p_n, q_n: n < \omega\}$ where the only relations are $p_n < q_m$ for $n \leq m$, and where $\delta_3$ is identically 2.
\end{definition}

\begin{theorem}
Suppose $(P, \leq, \delta) \in \mathbb{P}$ is not nearly binary crosscutting. Then there is $i < 4$ such that $\mathcal{C}(\Phi_{P_i}) \leq_B \mathcal{C}(\Phi_P)$. 
\end{theorem}
\begin{proof}
    We can suppose $P$ does not admit an $\omega$-chain, as otherwise $\mathcal{C}(\Phi_{P_0}) \leq_B \mathcal{C}(\Phi_P)$ by Lemmas~\ref{ReductionP}, ~\ref{ReductionDelta}. 

    Since $P$ is not nearly binary crosscutting, there are either infinitely many $p$ with $\delta(p) > 2$, or else there are infinitely many $p \in P$ which are nonmaximal. If there are infinitely many $p \in P$ with $\delta(p) > 2$ then by the nonexistence of an $\omega$-chain, we can find an infinite antichain on which $\delta$ is at least three, and so by Lemmas~\ref{ReductionP}, ~\ref{ReductionDelta} we get $\mathcal{C}(\Phi_{P_1}) \leq_B \mathcal{C}(\Phi_{P})$. So it suffices to consider the case where there is an infinite set $E$ of nonmaximal elements. Enumerate $E = \{p_n: n < \omega\}$. Choose $q_n > p_n$ for each $n$. After applying Ramsey's theorem and using $P$ has no $\omega$-chain, we can suppose that $\{p_n: n< \omega\}$ and $\{q_n: n < \omega\}$ are disjoint antichains.  After applying Ramsey's theorem again we can suppose that each of the following is either true for all $n < m$ or for no $n < m$:

\begin{itemize}
    \item $p_n < q_m$;
    \item $q_n < p_m$.
\end{itemize}
If each $q_n < p_m$ then each $q_n < q_m$ so we get an $\omega$-chain, contrary to hypothesis. If each $p_n < q_m$ then by Lemmas~\ref{ReductionP}, ~\ref{ReductionDelta} we get $\mathcal{C}(\Phi_{P_3}) \leq_B \mathcal{C}(\Phi_P)$; otherwise the same holds for $\mathcal{C}(\Phi_{P_2})$.
\end{proof}

Thus to prove Theorem~\ref{PThm} it suffices to show that $\mathcal{C}(\Phi_{P_i}): i < 4$ all have non-Borel isomorphism relation.

$\mathcal{C}(\Phi_{P_0})$ has non-Borel isomorphism relation by Theorem 5.6 of \cite{URL}. $\mathcal{C}(\Phi_{P_1})$ has non-Borel isomorphism relation by Theorem~\ref{TameExpThm} since the action of $S_3$ on a three-element set is not free. In this case we actually have $\mbox{sr}(\mathcal{C}(\Phi_{P_1})) = (2^{\aleph_0})^+$.

We consider $P_2 = \{p_n: n < \omega\} \cup \{q_n: n < \omega\}$ where the only relations are $p_n \leq q_n$ for each $n$. Let $\overline{M} = (M_n: n < \omega)$ where each $M_n$ is a set of size $4$ equipped with an equivalence relation with two classes, each of size $2$. Then $\mathfrak{M}_{P_2} \cong \prod_n M_n$; in particular they have the same first-order theory, so $T_{P_2}$ has Borel isomorphism relation if and only if $T_{\overline{M}}$ does. But the latter fails, because the action of $\mbox{Aut}(M_n)$ on $M_n$ is not free.

In particular we also have $\mbox{sr}(\mathcal{C}(\Phi_{P_2})) = (2^{\aleph_0})^+$. 

So to classify which $(P, \delta)$ have Borel isomorphism, it is enough to show that $(P_3, \leq, \delta_3)$ has non-Borel isomorphism. This is the subject of the next section.

\section{$(P_3, \leq, \delta_3)$}
\begin{theorem}
    $\mathcal{C}(\Phi(P_3, \leq, \delta_3))$ has non-Borel isomorphism relation. In fact, $\mathcal{U}(\Phi_{P_3})$ has non-Borel isomorphism relation, where $\mathcal{U}$ means we adjoin a single unary predicate.
\end{theorem}
The rest of this section is a proof. 

Recall $P_3 = \{p_n, q_n: n < \omega\}$ where the only relations are $p_n \leq q_m$ for all $n \leq m$, and $\delta_3$ is identically two. It is enough to construct, for every $\alpha < \omega$, pairs of (not necessarily countable) models $M, N \models \mathcal{U}(\Phi_{P_3})$ such that $M \equiv_{\alpha \omega} N$ but $M \not \equiv_{\infty \omega} N$. It greatly simplifies the limit stage of the proof to construct uncountable models rather than countable models, even though in theory one can apply downward Lowenheim-Skolem to get back to the countable world. 

We view $2^\omega$ as a group, where $(f+g)(n) =f(n) + g(n)$ mod $2$. Viewed as such, $2^\omega$ is a vector space over the field $\mathbb{F}_2$ with two elements.

Let $\MM$ be the isomorphic copy of $\MM_{P_3}$ with universe $2^\omega \times 2^\omega$, where $(f, g) E_{p_n} (f', g')$ iff $f(n) = f'(n)$, and $(f, g) E_{q_n} (f', g')$ iff $f \restriction_{n+1} = f' \restriction_{n+1}$ and $g(n) = g'(n)$. 

Let $\mathbf{C}$ be the set of all $C \subseteq 2^\omega \times 2^\omega$ such that for all $f \in 2^\omega$, $\{g \in 2^\omega: (f, g) \in C\}$ is a coset of a group. We denote this coset as $C[f]$. 

Say that $A \subseteq \MM$ is coherent if for all $(f, g), (f', g') \in A$, and for all $n$, if $f\restriction_n = f' \restriction_n$ then $g \restriction_n = g' \restriction_n$ (so in particular if $f = f'$ then $g= g'$). 

Suppose $C \in \mathbf{C}$ and $A \subseteq C$ is coherent. Then define $\mbox{rnk}^C(A)$ as follows: $\mbox{rnk}^C(A) \geq \alpha$ if and only if for all $\beta < \alpha$ and for all $f \in 2^\omega$, there is some $g$ such that $A \cup \{(f, g)\}$ is a coherent subset of $C$ of rank at least $\beta$. Clearly $A \subseteq B$ implies $\mbox{rnk}^C(A) \geq \mbox{rnk}^C(B)$.

Mainly we are only interested in $\mbox{rnk}^C(A)$ for $A$ finite.

\begin{lemma}
    Suppose $C \in \mathbf{C}$. Then for all finite coherent sequences $A \subseteq C$, $\mbox{rnk}^C(A)$ is the minimum of $\mbox{rnk}^C(\{(f, g)\})$ over all $(f, g) \in A$. Thus, for all $(f, g) \in C$, $\mbox{rnk}^C(f, g) \geq \alpha$ if and only if for all $f'$ and for all $\beta < \alpha$ there exists $g' \in C[f']$ such that $\{(f, g), (f', g')\}$ is coherent and $\mbox{rnk}^C(f', g') \geq \beta$.
\end{lemma}
\begin{proof}
    Straightforward.
\end{proof}

\begin{lemma}
    For every ordinal $\alpha$ and for every $n < \omega$ there is a formula $\phi_\alpha(x_i: i < n)$ of $\mathcal{L}_{\infty \omega}$ such that for all $C \in \mathbf{C}$,  and for all $A \in \MM^n$, $(\MM, C) \models \phi_\alpha(A)$ if and only if $A\subseteq C$ is coherent and $\mbox{rnk}^{C}(A) \geq \alpha$.
\end{lemma}
\begin{proof}
Straightforward. Here, by $(\mathfrak{M}, C)$ we mean the expansion of $\mathfrak{M}$ by a unary predicate for $C$.
\end{proof}

Suppose $C \in \mathbf{C}$. Then let $G \in \mathbf{C}$ be defined by $G[f] = $ the group of which $C[f]$ is a coset. It is clear that $\mbox{rnk}^G(\emptyset) =  \infty$, indeed $\{(f, \overline{0}): f \in 2^\omega\}$ is a total coherent sequence. It follows from the preceding that if $\mbox{rnk}^C(\emptyset) < \infty$ then $(\MM, C) \not \equiv_{\infty} (\MM, G)$. On the other hand,

\begin{lemma}
    Let $C, G$ be as above. If $\mbox{rnk}^C(\emptyset) \geq \alpha$ then $(\MM, C) \equiv_{\alpha} (\MM, G)$.
\end{lemma}
\begin{proof}
    For each $D \subset 2^\omega$ finite let $X_D = D \times 2^\omega \subseteq \MM$. Let $\Pi_D$ be the set of all partial isomorphisms $\pi$ from $(\MM, G) \to (\MM, C)$ with domain and range equal to $X_D$, such that $\pi$ fixes the first coordinate, $\pi(f, g) = (f, g')$. For each $\pi \in \Pi_D$ let $\mbox{rnk}(\pi)$ denote $\mbox{rnk}^C(A_\pi)$ where $A_\pi = \{\pi((f, \overline{0})): f \in D\}$. Then we know $\mbox{rnk}(\emptyset) \geq \alpha$. It suffices to show that for all $\beta$, whenever $\pi \in X_D$ has $\mbox{rnk}(\pi) = \beta$, then for all $f \in 2^\omega$ and for all $\gamma < \beta$ there is $\pi' \in X_{D \cup \{f\}}$ extending $\pi$ and of rank at least $\gamma$. This follows from the definition of $\mbox{rnk}^C$.
\end{proof}

It thus suffices to construct, by induction on $1 \leq \alpha< \omega_1$, some $C \in \mathbf{C}$ with $\mbox{rnk}^C(\emptyset) = \alpha$. 

When $\alpha = 1$ let $C[f] = \{\overline{0}\}$ unless $f = \overline{1}$, in which case $C[f] = \{\overline{1}\}$. Then it is clear that $\mbox{rnk}^C(\{(\overline{1}, \overline{1})\}) = 0$ so $\mbox{rnk}^C(\emptyset) = 1$.

Suppose $\alpha < \omega_1$ and we have constructed $C \in \mathbf{C}$ with $\mbox{rnk}(C) = \alpha$. Let $G$ be defined by $G[f] = $ the group of which $C[f]$ is a coset. Let $D \in \mathbf{C}$ be defined by the following clauses:

\begin{itemize}
    \item $D[0f] = \{\overline{0}\}$;
    \item For each $i < 2$, letting $j = 1 -i$, let $D[1i f] = i0 G[f] \cup j0 C[f]$.
\end{itemize}

Here, by $s X$ we mean the set of concatenations $\{st: t \in X\}$. If $X$ is a coset of a group $G$ then $sX$ is a coset of $0^{|s|} G$. Thus each $D[f]$ is seen to be a coset of a group, since the union of any two cosets of a fixed group is itself a coset of a group (since we are working with $\mathbb{F}_2$-vector spaces). We need to show $\mbox{rnk}(D) = \alpha+1$. 

For each $i < 2$ let $A_i$ be the coherent sequence consisting of all pairs $(0f, \overline{0})$ for $f \in 2^\omega$, and all pairs $(1if, i0\overline{0})$ for $f \in 2^\omega$.

Given $B \subseteq C$, and given $i < 2$, let $j = 1-i$ and let $F_i(B) = \{(1j f, i0g): (f, g) \in B\}$.

\begin{lemma}
Suppose $B \subseteq C$ and $i < 2$, Then $B$ is coherent if and only if $F_i(B)$ is coherent, which is the case if and only if $F_i(B) \cup A_i$ is coherent.
\end{lemma}
\begin{proof}
Straightforward.
\end{proof}

\begin{lemma}
    Suppose $B \subseteq C$ is coherent and $i < 2$. Then $\mbox{rnk}^D(A_i \cup F_i(B)) \geq \mbox{rnk}^C(B)$. If $B$ is nonempty then $\mbox{rnk}^C(B) \geq \mbox{rnk}^D(F_i(B))$. 
\end{lemma}
\begin{proof}
Let $j = 1-i$.

    First we show by induction on $\beta$ that $\mbox{rnk}^C(B) \geq \beta$ implies $\mbox{rnk}^D(A_i \cup F_i(B)) \geq \beta$. Suppose we have established this for all $\gamma < \beta$, and $B$ is given. Suppose $\tilde{f} \in 2^\omega$ is given. If $\tilde{f}(0) = 0$ or if $\tilde{f}(1) = i$ then $\tilde{f} \in \mbox{dom}(A_i)$ and there is nothing to do. So suppose $\tilde{f} = 1j f$. Choose $g$ such that $B \cup (f, g)$ is coherent of rank at least $\gamma$. Then $F_i(B \cup \{(f, g)\}) = F_i(B) \cup \{(1jf, i0g)\}$ is coherent by the preceding lemma, and is of rank at least $\gamma$ by the inductive hypothesis. 

    Next we show by induction on $\beta$ that for nonempty $B$, $\mbox{rnk}^D(F_i(B)) \geq \beta$ implies $\mbox{rnk}^C(B) \geq \beta$. So suppose we have established this for all $\gamma < \beta$, and $B$ is given. Let $f \in 2^\omega$. Choose $\tilde{g}$ such that $F_i(B) \cup \{(1jf, \tilde{g})\}$ is coherent of rank at least $\gamma$. Since $B$ is nonempty, we must have that $\tilde{g}$ is of the form $i0g$ for some $g \in C[f]$, and $\mbox{rnk}^C(B \cup \{(f, g)\}) \geq \gamma$ by the inductive hypothesis.
\end{proof}

\begin{lemma}
    $\mbox{rnk}^D(\emptyset) = \alpha+1$.
\end{lemma}
\begin{proof}
    First we show $\mbox{rnk}^D(\emptyset) \geq \alpha+1$. First note that for all $f \in 2^\omega$ there is $i < 2$ and $g \in 2^\omega$ such that $(f, g) \in A_i$; since $\mbox{rnk}^D(A_i) \geq \mbox{rnk}^C(\emptyset) = \alpha$, this implies $\mbox{rnk}^D(\{(f, g)\}) \geq \alpha$, so $\mbox{rnk}^D(\emptyset) \geq \alpha+1$.

    Conversely, we show that $\mbox{rnk}^D(\emptyset) \leq \alpha+1$. For this it suffices to show that for all $\tilde{g} \in D[10\overline{0}]$, $\mbox{rnk}^D(\{(10\overline{0}, \tilde{g})\}) \leq \alpha$. Write $\tilde{g} = j0g$. If $j = 1$ then this follows from the preceding lemma, since $\{(10 \overline{0},  10g)\} = F_1(\{(f, g)\})$. So suppose $j = 0$. Choose $f' \in 2^\omega$ such that for all $g' \in C[f']$, $\mbox{rnk}^C(\{(f', g')\}) < \alpha$. Then whenever $B := \{(10 \overline{0}, 00g), (11f', \tilde{g}')\}$ is coherent, we must have that $\tilde{g}'$ is of the form $00 g'$ for some $g' \in C[f']$. Hence $\mbox{rnk}^D(B) \leq \mbox{rnk}^D(11 f', 00g')$ which in turn by the preceding lemma is $\leq \mbox{rnk}^C(f', g') < \alpha$. 
\end{proof}

Now we describe how to handle limits. Let $C_n \in \mathbf{C}$ for $n < \omega$ be such that $\mbox{rnk}^{C_n}(\emptyset)$ is strictly increasing with limit $\alpha < \omega_1$. Let $G_n[f]$ be the group of which $C_n[f]$ is a coset, so each $G_n \in \mathbf{C}$.

Partition $\omega$ into infinite pieces $I, I_n: n < \omega$. We can suppose each $\min(I_n)$ is bigger than the $n$th element of $I$.

As notation, given infinite $J \subseteq \omega$ and $f \in 2^\omega$, let $f_J: \omega \to \omega$ be defined by $f_J(n) = f(j_n)$ where $J = \{j_m: m < \omega\}$ is enumerated in increasing order.

For each $f \in 2^\omega$ let $G[f]$ be the group of all $g \in 2^\omega$ such that for each $n$ if $g_I(n) = 0$ then $g_{I_n} \in G_n[f_{I_n}]$, and if $g_I(n) = 1$ then $g_{I_n} \in C_n[f_{I_n}]$. We have a homomorphism $\rho: G[f] \to 2^\omega$ sending $g$ to $g_I$. Let $H[f]$ be the preimage under $\rho$ of all eventually $0$ sequences, and let $D[f]$ be the preimage under $\rho$ of all eventually $1$ sequences, so $D[f]$ is a coset of $H[f]$.

It is convenient to write $C^0_n = G_n$, $C^1_n = C_n$.

It suffices to establish that $\mbox{rnk}^D(\emptyset) = \alpha$. Let $(f, g) \in D$. Then define $\tau(f, g)$ to be the minimum over all $n$ of $\beta_n$, where $\beta_n = \mbox{rnk}^{C^i_n}(f_{I_n}, g_{I_n})$, where $i = g_I(n)$.

\begin{lemma}
    $\tau(f, g) = \mbox{rnk}^D(\{(f, g)\})$. 
\end{lemma}
\begin{proof}
First we show by induction on $\beta$ that $\tau(f, g) \geq \beta$ implies $\mbox{rnk}^D(\{(f, g)\}) \geq \beta$. Suppose we have verified for all $\gamma < \beta$. Let $f' \in 2^\omega$ and $\gamma < \beta$ be arbitrary. For each $n$, choose $g'_n$ such that $\{(f_{I_n}, g_{I_n}), (f'_{I_n}, g'_n)\}$ is coherent and $\mbox{rnk}^{C^{g_I(n)}_n}(f'_{I_n}, g'_n) \geq \gamma$. Define $g' \in 2^\omega$ by $g'_{I_n} = g'_n$ and $g'_I = g_I$. Then $\{(f, g), (f', g')\}$ is easily seen to be coherent, and visibly $\tau(f', g') \geq \gamma$, so we conclude by the inductive hypothesis.

Second we show by induction on $\beta$ that $\mbox{rnk}^D(\{f, g\}) \geq \beta$ implies $\tau(f, g) \geq \beta$. Suppose we have verified for all $\gamma < \beta$. Let $n < \omega$ be given; we need to show $\mbox{rnk}^{C^{g_I(n)}_n}(f_{I_n}, g_{I_n}) \geq \beta$. Let $h \in 2^\omega$ and $\gamma < \beta$ be arbitrary. Define $f' \in 2^\omega$ by $f'_{I} = f_I$ and $f'_{I_m} = f_{I_m}$ for $m \not= i$ and $f'_{I_n} = h$. Choose $g'$ such that $\{(f, g), (f', g')\}$ is coherent and $\mbox{rnk}^D(f', g') \geq \gamma$. By the inductive hypothesis, $\mbox{rnk}^{C^{g_I(n)}_n}(f'_{I_n}, g'_{I_n}) \geq \gamma$, and it is easily seen that $\{(f_{I_n}, g_{I_n}), (f'_{I_n}, g'_{I_n})\}$ is coherent.
\end{proof}

\begin{lemma}
    $\mbox{rnk}^D(\emptyset) = \alpha$.
\end{lemma}
\begin{proof}
    Given any $(f, g) \in D$, there must be some $n$ with $g_I(n) = 1$ (since $g_I$ is eventually $1$), hence $\tau(f, g) \leq \mbox{rnk}^{C_n}(\emptyset) < \alpha$, so we get $\mbox{rnk}^D(\emptyset) \leq \alpha$. Conversely, suppose $\beta < \alpha$ and $f$ is given. Choose $n$ large enough so that $\mbox{rnk}^{C_n}(\emptyset) > \beta$. Choose $g \in 2^\omega$ so that $g_I = (0)^n \overline{1}$ and $g_{I_m} = \overline{0}$ for $m < n$ and such that for $m > n$, $\mbox{rnk}^{C_m}(f_{I_m}, g_{I_m}) \geq \beta$. Then by the preceding lemma, $\mbox{rnk}^D(f, g) \geq \beta$.
\end{proof}

\section{Bounding Scott Ranks}

The purpose of this section is to show $\mbox{sr}(T_{P_0}) = \mbox{sr}(T_{P_3}) = \omega_1$. We have already shown that the Scott ranks are at least $\omega_1$, so what remains to show is that every model has countable Scott rank.

\begin{definition}
    Suppose $M$ is a structure. Let $\mathcal{E}$ be a family of quantifier-free definable equivalence relations on $M$, and let $M^{\mathcal{E}}$ denote the expansion of $M$ where we add disjoint sorts $U_E$ for each $E \in \mathcal{E}$ and projections $\pi_E: M \to U_E$ such that for all $a, b \in M$, $M \models E(a, b)$ if and only if $\pi_E(a) = \pi_E(b)$. We identify each $U_E$ with $M/E$.
\end{definition}

\begin{lemma}\label{SRLemma7}
    Suppose $M, \mathcal{E}$ are as above. Then for all $\overline{a}, \overline{b} \in M^{<\omega}$ and for all ordinals $\alpha$, $(M, \overline{a}) \equiv_\alpha (M, \overline{b})$ if and only if $(M^{\mathcal{E}}, \overline{a}) \equiv_\alpha (M^{\mathcal{E}}, \overline{b})$. 
\end{lemma}
\begin{proof}
It is straightforward to check that for all ordinal $\alpha$, if $(M^{\mathcal{E}}, \overline{a}) \equiv_\alpha (M^{\mathcal{E}}, \overline{b})$ then $(M, \overline{a}) \equiv_\alpha (M, \overline{b})$.

The reverse direction is also an induction on $\alpha$.
For $\alpha = 0$ we use that each $E \in \mathcal{E}$ is quantifier-free definable. Suppose we have verified the claim for all $\beta < \alpha$. Suppose $(M, \overline{a}) \equiv_\alpha (M, \overline{b})$ and $a \in M^{\mathcal{E}}, \beta < \alpha$ are given. If $a$ is in the home sort of $M^{\mathcal{E}}$ then let $a' = a$. Otherwise, let $a'$ be some element of $M$ with $a = [a']_{E}$ for some $E \in \mathcal{E}$. Choose $b'$ with $(M, \overline{a}a') \equiv_\beta (M, \overline{b}b')$. By the inductive hypothesis, $(M^{\mathcal{E}}, \overline{a} a') \equiv_\beta (M^{\mathcal{E}}, \overline{b} b')$. By Lemma~\ref{SRLemma3}, $(M^{\mathcal{E}}, \overline{a}a'a) \equiv_\beta (M^{\mathcal{E}}, \overline{b} b' b)$ as desired.
\end{proof}

\begin{theorem}
    Every model of $T_{P_0} = \mbox{REF(bin)}$ has countable Scott rank.
\end{theorem}
\begin{proof}
It sufficecs to show the same for $\mathcal{C}(\Phi_{P_0})$. 

    Let $M \models \Phi_{P_0}$, and let $c: M \to \omega$. It suffices to show that $N := (M, c)$ has countable Scott rank.

    Let $\mathcal{E} = \{E_{n}: n < \omega\}$ be the refining equivalence relations. Let $\alpha$ be a large enough countable ordinal such that for all $n < \omega$ and for all $u, v\in M/E_n$, if $(N^{\mathcal{E}}, u) \equiv_\alpha (N^{\mathcal{E}}, v)$ then $(N^{\mathcal{E}}, u) \equiv_{\omega_1} (N^{\mathcal{E}}, v)$. We show that $\mbox{sr}(N) \leq \alpha$.

    Towards this, define a relation $\sim$ on pairs of tuples from $N$ by $\overline{a} \sim \overline{b}$ if they are of the same length, say $i_*$, and $\mbox{qftp}^N(\overline{a}) = \mbox{qftp}^N(\overline{b})$ and for each $i< i_*$, and each $(N, a_i) \equiv_\alpha (N, b_i)$. It suffices to show that $\sim$ is a back-and-forth relation.

    So suppose $\overline{a} \sim \overline{b}$ and $a \in N$ is given. Let $n$ be largest such that there is $i < i_*$ with $N \models E_n(a,a_i)$; note that this $i$ is unique, since $\delta_0$ is identically 2, i.e. we have binary splitting. (If we did not have binary splitting a similar argument would work, with some extra care.)

    By Lemma~\ref{SRLemma7} we have $(N^{\mathcal{E}}, a_i) \equiv_\alpha (N^{\mathcal{E}}, b_i)$. By Lemma~\ref{SRLemma3} we have $(N^{\mathcal{E}}, [a_i]_{E_{n+1}}) \equiv_\alpha (N^{\mathcal{E}}, [b_i]_{E_{n+1}})$. By choice of $\alpha$ we have $(N^{\mathcal{E}}, [a_i]_{E_{n+1}}) \equiv_{\alpha+1} (N^{\mathcal{E}}, [b_i]_{E_{n+1}})$. Choose $b \in M$ so that $(N^{\mathcal{E}}, [a_i]_{E_{n+1}} a) \equiv_\alpha (N^{\mathcal{E}}, [b_i]_{E_{n+1}}, b)$. Then this choice of $b$ works.
\end{proof}

\begin{theorem}
Every model of $T_{P_3}$ has countable Scott rank.
\end{theorem}
\begin{proof}
Suppose $M \models \Phi_{P_3}$, $c: M \to \omega$ is a coloring and $a_* \in M$. We can suppose $M$ as a submodel of $\mathfrak{M} = (2^\omega \times 2^\omega, \ldots)$ as before.  Let $N = (M, c)(a_*)$ where we add $c$ as a coloring and $a_*$ as a constant. By Lemma~\ref{SRLemma6} it suffices to show that $\mbox{sr}(N) < \omega_1$.

For each $n$, let $E_n$ be the equivalence relation on $N$ defined by $E_n(a, b)$ if and only if $E_{q_m}(a, b)$ for all $m < n$. Equivalently, writing $a = (f, g)$ and $b = (f', g')$, we have $E_n(a, b)$ holds if and only if $f \restriction_n = f' \restriction_n$ and $g \restriction_n = g' \restriction_n$. Let $\mathcal{E} = \{E_n: n < \omega\}$. 

Let $\alpha$ be a large enough countable ordinal so that $\alpha \geq \omega$ and for all $n < \omega$ and for all $u, v \in M/E_n$, if $(N^{\mathcal{E}}, u) \equiv_{\alpha} (N^{\mathcal{E}}, v)$ then $(N^{\mathcal{E}}, u) \equiv_{\omega_1} (N^{\mathcal{E}}, v)$. We claim that $\mbox{sr}(N) \leq \alpha$. 

Let $\sim$ be the relation on $N^{<\omega} \times N^{<\omega}$ defined by putting $\overline{a} \sim \overline{b}$ if they are both tuples of the same length, say $i_*$, and $\mbox{qftp}^N(\overline{a}) = \mbox{qftp}^N(\overline{b})$ and for all $i < i_*$, $(N, a_i) \equiv_\alpha (N, b_i)$. It suffices to show that $\sim$ is a back-and-forth system on $N$.

Suppose $\overline{a} \sim \overline{b}$ and $a$ is given. We can suppose $a_*$, the constant we added to $N$, is equal to $a_0 = b_0$, since prepending $a_*$ to both $\overline{a}$ and $\overline{b}$ does not disturb $\sim$ by Lemma~\ref{SRLemma3}. For each $i < i_*$ write $a_i = (f_i, g_i)$ for $f_i, g_i \in 2^\omega$. Then since $\mbox{qftp}^N(a_0, a_i) = \mbox{qftp}^N(b_0, b_i)$ and $a_0 = b_0$ we necessarily have $b_i$ is of the form $(f_i, g'_i)$ for some $g'_i \in 2^\omega$. Write $a= (f, g)$. There are two cases.

First, suppose $f \in \{f_i: i < i_*\}$. Say $f = f_i$. Then $a$ is the unique realization of its quantifier-free type over $a_i$, so by Lemma~\ref{SRLemma4} there is a (unique) $b$ with $(N, a_i, a) \equiv_\alpha (N, b_i, b)$. In particular $(N, a) \equiv_\alpha (N, b)$. It is easily checked that $\mbox{qftp}^N(\overline{a}a) = \mbox{qftp}^N(\overline{b} b)$ so $\overline{a}a \sim \overline{b} b$ as desired.

Second, suppose $f \not \in \{f_i: i < i_*\}$. Let $n$ be largest such that $f \restriction_n \in \{f_i \restriction_n: i < i_*\}$; say $f \restriction_n = f_i \restriction_n$.

By Lemma~\ref{SRLemma7} we have $(N^{\mathcal{E}}, a_i) \equiv_\alpha (N^{\mathcal{E}}, b_i)$. By Lemma~\ref{SRLemma3} we have $(N^{\mathcal{E}}, [a_i]_{E_n}) \equiv_\alpha (N^{\mathcal{E}}, [b_i]_{E_n})$. By choice of $\alpha$, $(N^{\mathcal{E}}, [a_i]_{E_n}) \equiv_{\alpha+1} (N^{\mathcal{E}}, [b_i]_{E_n})$. Thus we can find some $b \in N$ (i.e. in the home sort of $N^{\mathcal{E}}$) such that $(N^{\mathcal{E}}, [a_i]_{E_n}, a) \equiv_\alpha (N^{\mathcal{E}}, [b_i]_{E_n}, b)$. Thus $(N^{\mathcal{E}}, a) \equiv_\alpha (N^{\mathcal{E}}, b)$. By Lemma~\ref{SRLemma7} we have $(N, a) \equiv_\alpha (N, b)$, and it is easily checked that $\mbox{qftp}^N(\overline{a}a) = \mbox{qftp}^N(\overline{b} b)$ so $\overline{a}a \sim \overline{b} b$ as desired.
\end{proof}


\end{document}